\documentclass[11pt]{aims}

\usepackage{amsmath, amssymb, mathrsfs}

  \usepackage{paralist}
  \usepackage{graphics} %% add this and next lines if pictures should be in esp format
  \usepackage{epsfig} %For pictures: screened artwork should be set up with an 85 or 100 line screen
\usepackage{graphicx} 
 \usepackage{epstopdf}%This is to transfer .eps figure to .pdf figure; please compile your paper using PDFLeTex or PDFTeXify.
 \usepackage[colorlinks=true]{hyperref}
 \hypersetup{urlcolor=blue, citecolor=blue}
 \usepackage[toc,page]{appendix}
\usepackage{chngcntr}
\usepackage{hyperref}

\usepackage{titlesec}
\titleformat{\section}{\Large\bfseries}{\thesection.}{4pt}{}
\titleformat{\subsection}{\large\bfseries}{\thesection.\arabic{subsection}.}{4pt}{}
\titleformat{\subsubsection}{\bfseries}{\thesection.\arabic{subsection}.\arabic{subsubsection}.}{4pt}{}
\titleformat*{\paragraph}{\bfseries}
\titleformat*{\subparagraph}{\bfseries}
\usepackage[margin=1in]{geometry}

  \def\currentyear{}
   \def\currentmonth{}
    \def\ppages{}
     \def\DOI{}

\newtheorem{theorem}{Theorem}[section]

\newtheorem{prop}{Proposition}[section]
\newtheorem{lemma}[prop]{Lemma}
\newtheorem{definition}[prop]{Definition}
\newtheorem{cl}[prop]{Claim}

\newtheorem{rem}[prop]{Remark}

\def\div{{\rm div}\,}% divergence
\def\qed{\hbox {\hskip 1pt \vrule width 4pt height 6pt depth 1.5pt
\hskip 1pt}\\} % cqfd

\def\T0{T_{0,1}}

\def\sf{s_{0,3}}

\def\sc {s_{0,1}}

\def\tt {{\varrho}}

\newcommand {\R}{ \mathbb{R}}

\newcommand {\N}{ \mathbb{N}}

\newcommand {\pa}{\partial}

\newcommand {\beqna} {\begin{eqnarray}}
\newcommand {\eeqna} {\end{eqnarray}}
\newcommand {\beqtn} {\begin{equation}}
\newcommand {\eeqtn} {\end{equation}}

\newcommand {\dsp}{\displaystyle}

\def\and {{\rm \; and \;}}

\def\exp {{\rm exp}}

\numberwithin{equation}{section}

\title{Modulation theory for the flat blow-up solutions of nonlinear heat equation}

\author[G. K. Duong,  N. Nouaili and H. Zaag ]{}


\thanks{\today}
\begin{document}
\maketitle

% Enter the first author's name and address:

\centerline{Giao Ky Duong$^{(1),(2)}$,   Nejla Nouaili$^{(3)}$   and Hatem Zaag$^{(4)}$} 
\medskip
{\footnotesize
\centerline{ $^{(1)}$ Faculty of Education, An Giang University, Vietnam.}
\centerline{ $^{(2)}$ Vietnam National University, Ho Chi Minh City, Vietnam. }
  \centerline{ $^{(3)}$ CEREMADE, Universit\'e Paris Dauphine, Paris Sciences et Lettres, France }
   \centerline{ $^{(4)}$Universit\'e Sorbonne Paris Nord,
LAGA, CNRS (UMR 7539), F-93430, Villetaneuse, France.}
}

\bigskip
\begin{center}\thanks{\today}\end{center}

\begin{abstract} 
In this paper, we revisit the proof of the existence of a solution to the semilinear heat equation in one space dimension with a \textit{flat} blow-up profile, already proved by Bricmont and Kupainen together with Herrero and Vel\'azquez. Though our approach relies on the well-celebrated method, based on the reduction of the problem to a finite-dimensional one, then the use of a topological "shooting method" to solve the latter, the novelty of our approach lays in the use of a modulation technique to control the projection of the zero eigenmode arising in the problem. Up to our knowledge, this is the first time where modulation is used with this kind of profiles. We do hope that this simplifies the argument.
%
% We construct a solution to the nonlinear heat equation, which blows up in finite time $T$ only at one blow-up point in one dimension. We also give a sharp description of its \textit{flat} blow-up profile. The proof relies on the reduction of the problem to a finite dimensional one, and the use of a topological "shooting argument'' combined with a priori estimates to rigorously find a solution close to the desired profile. The main novelty in this work is the use of modulation parameter.

\end{abstract}

\maketitle
\section{Introduction}\label{section-Introduction}

We consider the following nonlinear heat equation (NLH)
\begin{equation} 
\label{NLH}
 \left\{ \begin{array}{l}
u_t=\Delta u + |u|^{p-1}u,\\
u(.,0)=u_0\in L^\infty (\R^N,\R),
\end{array} \right.
\mbox{     }
\end{equation}
where $p>1$ and $u(x,t):\R^N\times [0,T)\to \R$. 
Equation \eqref{NLH} is considered as a model for many physical situations, such as heat transfer, combustion theory, thermal explosion, etc. (see more in Kapila \cite{KSJAM80}, Kassoy and Poland \cite{KPsiam80,KPsiam81},  Bebernes and Eberly  \cite{BEbook89}). 
%
% The problem in  \eqref{NLH} arises from the problem of heat transfer, combustion theory, thermal explosion, see more in Kapila \cite{KSJAM80}, Kassoy and Poland \cite{KPsiam80,KPsiam81},  Bebernes and Eberly  \cite{BEbook89}.  
Firstly, note that equation \eqref{NLH} is well-posed in $L^\infty$.
% Firstly, the Cauchy problem to \eqref{NLH} is well-posed in $L^\infty$. 
More precisely, for each $u_0 \in L^\infty(\R^N)$, one of the following statements holds:
\begin{itemize}
    \item either the solution is global in time;
    \item or the maximum existence time is finite i.e. $T_{max} <+\infty  $ and
    \begin{equation}\label{norm-infty-u-infty}
       \lim_{t \to T_{max}} \left\| u(\cdot, t) \right\|_{L^\infty} = +\infty.
    \end{equation}
\end{itemize}
 In particular,  $ T_{max} >0$    is called  the blowup  time of the solution and a point $a \in \mathbb{R}^N$ is called a blowup point of the solution 
 %if and only 
 if there exists sequence $(a_n,t_n) \to (a,T)$  as $  n \to +\infty$ such that 
$$ \left| u(a_n, t_n) \right| \to +\infty \text{ as } n \to + \infty. $$
\iffalse
We can see that we have the trivial blowup solution to \eqref{NLH} given as follows 
$$ \psi_T(t) = \kappa (T-t)^{-\frac{1}{p-1}} \text{ where } \kappa = (p-1)^{-\frac{1}{p-1}}.  $$
Accordingly  to the  classification investigated by Merle and Matano \cite{MMcpam04}, the blowup solutions of the equation \eqref{NLH} satisfy the following estimate
\begin{equation}\label{defi-type-I}
    \left\| u(\cdot, t) \right\|_{L^\infty(\mathbb R^N)} \le C \psi_T(t), \forall t \in [0,T),
\end{equation}
called by \textit{Type I} blowup solutions, otherwise, they are of \textit{ Type II}.  In the context  of the paper, we aim to study the \textit{Type I}. 
There is a huge of literature concerned by the study of blow-solution of  \textit{Type I}, we cite for example ...\\
\fi

\medskip

%As a matter of fact,  
Blowup for equation \eqref{NLH} has been studied intensively by many mathematicians and no list can be exhaustive.
This is the case for the question of deriving blowup profiles, which is completely understood in one space dimension (see in particular Herrero and Vel\'azquez \cite{HVasps92,HVdie92,HVaihn93,HVcras94}), unlike the higher dimensional case, where much less is known (see for example Vel\'azquez \cite{Vcpde92,VELtams93,VELiumj93}, Zaag \cite{ZAAaihp02,ZAAcmp02, Zdmj06,Zmme02} together with the recent contributions by Merle and Zaag \cite{MZimrn21,MZ22}).
%
%In particular, the question  on deriving the blowup profiles which is completely understood in one space dimension (however, less is understood in higher dimensions, see more details  \cite{HVasps92,HVaihn93,HVcras94} and \cite{ZAAaihp02,ZAAcmp02}).

\medskip

In the one dimensional case, Herrero and Vel\'azquez proved the following, unless the solution is space independent (see also Filippas and Kohn \cite{FKcpam92}):
%
% In the one dimensional case, one proved  one of the following statements holds  
\begin{itemize}
    \item[$(i)$] Either
    \begin{equation*}
        \sup_{|x-a| \le K \sqrt{(T-t)|\ln(T-t)|}} \left| (T-t)^{\frac{1}{p-1}}u(x,t) - \varphi\left(\frac{x-a}{\sqrt{(T-t)|\ln(T-t)|}} \right)    \right|  \to 0,
    \end{equation*}
     where $\varphi(z)=(p-1+b_g|z|^{2k})^{\frac{1}{p-1}}$ and $b_g= \frac{(p-1)^2}{4p}$    is unique.
(note that Herrero and Vel\'azquez proved that this behavior is generic in \cite{HVcras94,HVasnsp92}).
 %   In this case Herrero and Vel\'azquez conjectured that this profile is generic.
    \item[$(ii)$] Or  
     \begin{equation*}
        \sup_{|x-a| \le K (T-t)^{\frac{1}{2k}}} \left| (T-t)^{\frac{1}{p-1}}u(x,t) - \varphi_k\left(\frac{x-a}{(T-t)^{\frac{1}{2k}}} \right)    \right|  \to 0,
    \end{equation*}
    where $\varphi_k(z)=(p-1+b|z|^{2k})^{\frac{1}{p-1}}$, where $b$ is an \textit{arbitrary} positive number. 
    
\end{itemize}

In particular, we are interested in constructing blowup solution with a prescribed behaviors, via  a ``generic approximation'', called the blowup profile of the solution. \\

The existence of such solutions was observed by Vel\'{a}zquez, Galaktionov and Herrero \cite{VelGalHer91} who indicated formally how one might find these solution. Later, Bricmont and Kupiainen \cite{BKnon94}, will give a rigorous proof of construction of such profiles (see also Herrero and Vel\'azquez \cite{HVdie92} for the profile $\varphi_4$). In \cite{AVjfdram97}, Angenent and Vel\'{a}zquez gives a construction of blow up solution for the mean curvature flow inspied by the construrction of  (ii).
Most of the constructions are made in one dimension $N=1$. In higher dimension $N\geq 2$, recently Merle and Zaag give the construction of a new profile of type I with a superlinear power in the Sobolev subcritical range, for more details see \cite{MZ22}.

\medskip

In this paper we revisit the construction of ii) given in Section 4 of \cite{BKnon94}. Our construction has the advantage that it uses the modulation parameter.
We shall use a topological "shooting" argument to prove existence of Solutions constructed in Theorem \ref{Theorem-principal}. The construction is essentially an adaptation of Wazewski's principle (see \cite{Conbook78}, chapter II and the references given there). The use of topological methods in the analysis of singularities for blow-up phenomena seems to have been introduced by Bressan in  \cite{Breiumj90}.

\medskip
\noindent 
The following is the  main result in the paper

\begin{theorem} \label{Theorem-principal}Let  $p>1$ and $k \in \mathbb N, k \ge 2$, then there exist $\delta_0$ and $ \tilde T >0$ such that for all $\delta \in (0,\delta_0)$ and $T \in (0,\tilde T)$, we can construct initial datum $u_0 \in L^\infty(\R)$ such that the corresponding solution to equation \eqref{NLH} blowup in finite time $T$ and only at the origin. Moreover, there exists the flow $b(t) \in C^1(0,T)$ such that   the following description is valid\\
(i) For all $t\in [0,T)$, it holds that 
\begin{equation}\label{theorem-intermediate}
    \left \| (T-t)^{\frac{1}{p-1}} u(\cdot, t) - f_{b(t)}\left(\frac{|\cdot|^{2k}}{T-t} \right)\right \|_{L^\infty(\mathbb{R})}\lesssim  (T-t)^{\frac{\delta}{2}(1-\frac{1}{k})} \text{ as } s \to  \infty.
\end{equation}
(ii) There exists $b^*>0$ such that  $b(t)\to b^*$ as $t\to T$ and 
\begin{equation}\label{estimate-b-t-b-*}
\left|  b(t) - b^*  \right| \lesssim  (T-t)^{\delta (1-\frac 1k)}, \forall t \in (0,T),
\end{equation}
where $f_{b(t)}$ is defined by 
\begin{equation}
   f_{ b(t)}(y)= \left(p-1+ b(t) y^{2k}\right)^{-\frac{1}{p-1}}\;\;.
  \end{equation}
\end{theorem}
\begin{rem}
One of the most important steps of the proof is to project the linearized partial differential equation \eqref{equation-q} on the $H_m$, given by \eqref{eigenfunction-Ls}. We note that this is technically different from the work of Bricomont and Kupiainen \cite{BKnon94}, where the authors project the integral equation.
Consequently, we will have additional difficulty coming from the projection of the the different terms, see for example Lemma \ref{Lemma-Pn_partialq} and Lemma \ref{Lemma-P-n-mathcal-L-s}.
\end{rem}
\begin{rem}
We note that $ \frac{b_0}{2} \le b(t) \le 2 b_0
$ and \eqref{estimate-b-t-b-*}, it holds that
\begin{equation*}
\left\|  \left( p-1 + b(t) y^{2k} \right)^{-\frac{1}{p-1}} -   \left( p-1 + b^* y^{2k} \right)^{-\frac{1}{p-1}}   \right\|_{L^\infty(\mathbb{R})} \lesssim \left| b(t) - b^* \right| \lesssim (T-t)^{\delta\left( 1 -\frac{1}{k}\right)},
\end{equation*}
which yields 
\begin{equation}
\left \| (T-t)^{\frac{1}{p-1}} u(\cdot, t) - f_{b^*}\left(\frac{|\cdot|^{2k}}{T-t} \right)\right \|_{L^\infty(\mathbb{R})}\lesssim  (T-t)^{\frac{\delta}{2}(1-\frac{1}{k})} \text{ as } s \to  \infty.
\end{equation}
\end{rem}

The paper is organised as follows. In Section \ref{section-Formulation} and \ref{Section-Spectral-properties-Ls}, we give the formulation of the problem.
In Section \ref{section-Proof-assuming-estimates} we give the proof of the existence of the profile assuming technical details. In particular,
we construct a shrinking set and give an example of initial data giving rise to the blow-up
profile and at the end of the section we give the proof of Theorem \ref{Theorem-principal}.
The topological argument of Section \ref{section-Proof-assuming-estimates} uses a number of estimates given by Proposition \ref{proposition-ode}, we give the proof of this proposition in Section \ref{Section-proof-proposition-ode}. 

\textbf{Acknowledgement:} The author Giao Ky Duong is supported by the scientific research project of  An Giang University under the Grant 22.05.SP.

\section{Formulation of the problem}\label{section-Formulation}
Let us consider $T>0$, and $k \in \N, k \ge 2$, and $u$ be 
 a solution to \eqref{NLH} which blows up in finite time $T>0$. Then, we   introduce the following \textit{blow-up variable}: 
\begin{equation}\label{change-variable}
    w(y,s)=(T-t)^{-\frac{1}{p-1}}u(x,t),\;y=\frac{x}{(T-t)^{\frac{1}{2k}}},\;\; s=-\ln  (T-t).
\end{equation}
 Since  $u$ solves  \eqref{NLH}, for all $(x,t)\in\R^N\times[0,T)$,  then $w(y,s)$ reads the following equation
   \begin{equation}\label{equation-w}
       \frac{\partial   w}{\partial s}=I^{-2}(s) \Delta w - \frac{1}{2k} y  \cdot  \nabla  w -\frac{1}{p-1} w +|w|^{p-1}w,
   \end{equation}
 where $I(s)$ is defined by 
 \begin{equation}\label{defi-I-s}
     I(s) =  e^{\frac{s}{2}\left(1-\frac{1}{k} \right)}.
 \end{equation}
Adopting the \textit{setting} investigated by \cite{BKnon94},  we consider  $C^1$-flow $b $  and introduce    
\begin{equation}\label{decompose-equa-w-=q}
    w (y,s)=  f_{b(s)}(y) \left(1 + e_{b(s)}(y)q(y,s) \right)
\end{equation}
 where $f_b$ and  $e_b$  respectively defined as
 \begin{equation}\label{defi-profile}
   f_b(y)= \left(p-1+ b y^{2k}\right)^{-\frac{1}{p-1}},
  \end{equation}
 and 
 \begin{eqnarray}
 e_b(y) = \left( p-1 + b |y|^{2k}   \right)^{-1} \label{defi-e-b}.
 \end{eqnarray}   
and the flow $b$ will arise as an  unknown functions that will be constructed together with the linearized solution $q$. Since $f_b e_b=f_b^{p}$, by \eqref{decompose-equa-w-=q} $q$ can be written as follows
\begin{eqnarray}\label{decom-q-w-}
q=wf_b^{-p}-  (p-1+by^{2k}).
\end{eqnarray}

\noindent 
In the following we consider $(q,b)(s)$ which satisfies the following equation 
 
\beqtn\label{equation-q}
 \pa_s q =\mathcal{L}_s q+ \mathcal{N} (q) +\mathcal{D}_s(\nabla q)+\mathcal{R}_s(q) +b'(s)\mathcal{M}(q),
\eeqtn

where 
\begin{eqnarray}
\mathcal{L}_s q & = & I^{-2}(s) \Delta q-\frac{1}{2k}y \cdot \nabla q+q,\;\;\ I(s)=\dsp e^{\frac s2(1-\frac 1k)},\label{operator-Ls}\\
\mathcal{N}(q)&=&\left| 1+e_bq \right|^{p-1}(1+e_bq)-1-p e_b q \label{nonlinear-term}\\\mathcal{D}_s(\nabla q)&=&-\frac{4pkb}{p-1}I^{-2}(s) e_by^{2k-1}\nabla q, \label{equation-Ds} \\
\mathcal{R}_s(q)&=& I^{-2}(s)y^{2k-2} \left (\alpha_1+\alpha_2 y^{2k}e_b+(\alpha_3+\alpha_4 y^{2k}e_b)q \right), \label{equation-Rs} \\
\mathcal{M} (q) & = &\frac{p}{p-1}y^{2k} (1+ e_bq) \label{new-term}
\end{eqnarray}
and the constants $\alpha_i$ are given by 
\begin{equation}\label{defi-constant-in-R}
\begin{matrix}
\alpha_1 =-2k(2k-1)\frac{b}{p-1}; & \alpha_2=4pk^2\frac{b^2}{(p-1)^2}; & \alpha_3=-2pk(2k-1)\frac{b}{p-1};\alpha_4 =4p(2p-1)k^2\frac{b^2}{(p-1)^2} .\\
\end{matrix}
\end{equation}

\iffalse
Our goal is to prove  the following Proposition:
\begin{prop} There exists $ s_1<\infty$ and $\varepsilon >0$, such that for $s_0 > s_1$ and $g $ in $C^0(\mathbb{R})$ such that the equation \eqref{equation-w} with initial data \eqref{initial-data} has a unique classical solution, which satisfies
\begin{equation}\label{theorem-intermediate}
    \left \|w(.,s)- f_{b(s)}(.)\right \|_{\infty}\to 0\mbox{ as $s\to  \infty$},
\end{equation}
and
$$ b(s) = b(T, b_0, p, k) + O(e^{-s(k-1)}), \text{ and } b(T, b_0, p, k)  >0. $$

%$b(s)\to b_0$ as $s\to \infty$ and
\begin{equation}\label{defi-profile}
   f_b(y)= \left(p-1+ b y^{2k}\right)^{-\frac{1}{p-1}},\;\; k>1,
  \end{equation}
  and $f_b$ satisfy
  \begin{equation}
      0=-\frac k 2\nabla f_b-\frac{1}{p-1}f_b+|f_b|^{p-1}f_b.
  \end{equation}
\end{prop}

First we introduce the derivation of $w$ from $f_b$. It is convenient to write $w $ in the form 
\beqtn\label{definition-q}
w(y,s)=f_{b}(y) \left (1+e_b(y)q(y,s)\right),
\eeqtn
where,
\begin{equation}\label{defi-e-b}
   e_b(y)=\left (p-1+by^{2k}\right)^{-1} .
   \end{equation}
\fi 

\iffalse
\begin{rem} From \eqref{definition-q}, we can write
\[q=(w-f_b) (f_b e_b)^{-1}\]
we note that $f_b e_b=f_b^{p}$, then we obtain
\[q=(w-f_b) \left (p-1+by^{2k}\right )^{\frac{p}{p-1}}\]
\textcolor{blue}{\[q=wf_b^{-p}-\left (p-1+by^{2k}\right )\]
}
\en{rem}
\medskip
\fi

\section{Decomposition of the solution}\label{Section-Spectral-properties-Ls}

\subsection{ Fundamental solution involving to  $\mathcal{L}_s$}
Let us define   Hilbert space $L^2_{\rho_s}(\R)$  by 
\beqtn\label{define-L-2-rho-s} 
L^2_{\rho_s}(\R)=\{f \in L^2(\R),\; \int_{\R}f^2\rho_s dy<\infty\},
\eeqtn
where 
\begin{equation}\label{defi-rho-s}
    \displaystyle \rho_s=\frac{I(s)}{\sqrt{4\pi}} e^{-\frac{I_{s}^{2}y^2}{4}},
\end{equation}
and $I(s)$ is defined by \eqref{defi-I-s}.\\
In addition, we denote 
\beqtn\label{eigenfunction-Ls}
H_m(y,s)=I^{-m}(s)h_m(I(s) y)=\sum_{\ell=0}^{[\frac{m}{2}]}\frac{m!}{\ell!(m-2\ell)!}(-I^{-2}(s))^\ell y^{m-2\ell}
%H_m(y,s)==I^{-m}(s)h_m(I(s) y)=\sum_{\ell=0}^{[\frac{m}{2}]}\frac{m!}{\ell!(m-2\ell)!}(-1)^\ell I^{m-2\ell}(s) y^{m-2\ell},
\eeqtn
%\[
%H_m(y,s)=I^{-m}(s)h_m(I(s) y)=\sum_{\ell=0}^{[\frac{m}{2}]}\frac{m!}{\ell!(m-2\ell)!}(-I^{-2}(s))^\ell y^{m-2\ell},\]
where $h_m(z)$ be Hermite polynomial (physic version) 
\beqtn\label{definition-h-n-z}
h_m(z)=\sum_{\ell=0}^{[\frac{m}{2}]}\frac{m!}{\ell!(m-2\ell)!}(-1)^\ell z^{m-2\ell}.
\eeqtn
In particular, it is well known that
\[\int h_nh_m\rho_s dy=2^nn!\delta_{nm},\]
which yields
\beqtn\label{scalar-product-hm}
\dsp
({H}_n(.,s),{H}_m(.,s))_s=\int {H}_n(y){H}_m(y)\rho_s(y)dy=I^{-2n}2^n n!\delta_{nm}.
\eeqtn

\textbf{Jordan block's decomposition of $\mathcal{L}_s$}

\medskip

By a simple computation (relying on fundamental identities of Hermite polynomials), we have  
\beqtn\label{Ls-Hm}
\mathcal{L}_s H_m(y,s)=
\left\{
\begin{array}{lll}
& \left(1-\frac{m}{2k} \right)H_m(y,s)+m(m-1)(1-\frac{1}{k})I^{-2}(s)H_{m-2}&\mbox{ if }m\geq 2\\
& \left(1-\frac{m}{2k} \right)H_m(y,s)&\mbox{ if }m=\{0,1\}
\end{array}
\right. .
\eeqtn
We define   $\mathcal{K}_{s,\sigma}$ as  the fundamental solution to 
\beqtn
\pa_s \mathcal{K}_{s\sigma}=\mathcal{L}_s \mathcal{K}_{s\sigma}  \text{ for } s > \sigma   \mbox{ and }\mathcal{K}_{\sigma\sigma}=Id.
\eeqtn
By using the Mehler's formula, we can explicitly write its kernel as follows
\beqtn\label{Kernel-Formula}
\dsp \mathcal{K}_{s\sigma}(y,z)=e^{s-\sigma}\mathcal{F} \left ( e^{-\frac{s-\sigma}{2k}}y-z \right )
\eeqtn
where
\beqtn\label{Kernel-Formula-F}
\dsp \mathcal{F}(\xi)=\frac{L(s,\sigma)}{\sqrt{4\pi}}e^{-\frac{L^2(s,\sigma)\xi^2}{4}}\mbox{ where } L(s, \sigma) =\frac{I(\sigma)}{\sqrt{1-e^{-(s-\sigma)}}}\mbox{ and }I(s)=\dsp e^{\frac s2(1-\frac 1k)}.
\eeqtn
In addition, we have the following equalities
\beqtn
\mathcal{K}_{s\sigma}H_n(.,\sigma)=e^{(s-\sigma)(1-\frac{n}{2k})}H_n(.,s), n \ge 0.
\label{kernel-Hn}
\eeqtn
\iffalse

b) \textit{Multi-dimensional case:} Let $N \ge 2$ and the case is a natural extension of the setting  in the  one dimensional one. Indeed, we introduce  $L^2_{\rho_s}(\mathbb{R}^N)$  as in \eqref{define-L-2-rho-s} with 
$$ \rho_s (y) = \frac{I^N(s)}{(4\pi)^\frac{N}{2}} e^{- \frac{I^2(s)|y|^2}{4}}, y \in \mathbb{R}^N.$$\\
In addition, let $\alpha$ be a multi-index in $\mathbb{N}^N$, $\alpha = (\alpha_1,...,\alpha_N)$ and $|\alpha|= \alpha_1+...+\alpha_N$. Similarly the one dimensional case, we have Jordan's block's decomposition 
\begin{equation}
\mathcal{L}_s H_\alpha (y) = \left\{  \begin{array}{rcl}
\left( 1 - \frac{|\alpha|}{2k}  \right) H_\alpha(y) + 

\end{array}   \right. 
\end{equation}

Corresponding to eigenvalue $\lambda_m = 1 - \frac{m}{2k}$, eigenspace $\mathcal{E}_m$ is given by 
$$\mathcal{E}_m = \left\langle H_\alpha(y), |\alpha| =m \right\rangle,  $$
where $H_\alpha$ defined by 
\begin{eqnarray*}
H_\alpha(y,s) = \Pi_{i=1}^N  H_{\alpha_i}(y_i,s) \text{ with } H_{\alpha_i} \text{ given in } \eqref{eigenfunction-Ls}.
\end{eqnarray*}
In particular, semigroup $\mathcal{K}_{s,\sigma}$ has the same structure   as the first case that its kernel explicitly given by 
$$ \mathcal{K}_{s, \sigma}(y,z) =  \frac{e^{s- \sigma} L^N(s, \sigma) }{(4\pi)^\frac{N}{2}}  e^{-\frac{L^2(s,\sigma)}{4} \left|e^{-\frac{s -\sigma}{2k} y - z} \right|}.$$
\fi 
\subsection{ Decomposition of $q$.}
For the sake of controlling the unknown function $q \in L^2_{\rho_s}$,  we  will expand it with respect to the polynomials $H_m(y,s)$. We start by writing
\begin{equation}\label{decomposition-q2}
   q(y,s) =  \sum_{m=0}^{[M]}q_m(s) H_m(y,s)+ q_-(y,s) \equiv \dsp q_+(y,s)+q_-(y,s), 
\end{equation}
where constant  $[M]$ be the largest integer less than $M$ with  
\begin{equation}\label{defi-M}
  M=\frac{2kp}{p-1} . 
\end{equation}
From \eqref{scalar-product-hm}, we have
\beqtn\label{defi-q_m}
\begin{array}{rcl}
q_m(s) = P_m(q) =  \dsp \frac{\left\langle q,H_m \right\rangle_{L^2_{\rho_s}}}{\langle H_m,H_m\rangle_{L^2_{\rho_s}}}, 
\end{array}
\eeqtn
 as the projection of $q$ on  $H_m$. In addition,  $q_-(y,s)$ can be seen as the projection of $q$ onto $\{H_m,  m \ge [M]+1\}$ and  we also  denote as follow
\beqtn\label{projector-P-}
q_-=P_-(q). 
\eeqtn

\subsection{Equivalent  norms }
Let us consider  $L^\infty_M$ defined by 
\begin{equation}\label{defi-L-M}
    L^{\infty}_M(\R)=\{g \text{ such that } (1+|y|^M)^{-1} g \in L^\infty (\R)\},
\end{equation}
and $L^\infty_M$ is complete with the norm
\begin{equation}\label{defi-norm-L-M}
    \|g\|_{L^\infty_M} = \|(1+|y|^M)^{-1} g \|_{L^\infty},
\end{equation}
\iffalse 

Considering  $C^0(\R)$  which  we  introduce the  norms  for:  Let us consider  $q\in C^0(\R)$  with the decomposition in  \eqref{decomposition-q2}, \fi 
we  introduce
\beqtn\label{norm-q}
\|q\|_s=\sum_{m=0}^{[M]}|q_m|+|q_-|_s,
\eeqtn
where
\beqtn\label{norm-q-||-s}
|q_-|_s=\displaystyle \sup_{y}\frac{|q_-(y,s)|}{I(s)^{-M}+|y|^M}.
\eeqtn
It is straightforward to check that
\[C_1(s)\|q\|_{L^\infty_M} \leq \|q\|_s\leq C_2(s)\|q\|_{L^\infty_M}\mbox{ where }C_{i,i=1,2}(s) >0.\]
%and $\|q\|=\sup_{y}\frac{|q(y)|}{1+|y|^M}$.\\
%so $C^0(\R^N)$ is complete in the norm $\|.\|_s$.\\
\iffalse
In particular, we introduce $\|.\|_s$ as follows
\beqtn\label{norm-q-2}
\|q\|_s=\sum_{m=0}^{[M]}|q_m|+|q_-|_s,
\eeqtn
where
\beqtn\label{defi-|-|-s-norm}
|q_-|_s=\displaystyle \sup_{y}\frac{|q_-(y,s)|}{I(s)^{-M}+|y|^M}
\eeqtn
As a matter of fact, we have the following equivalence:
\[C_1(s)\|q\|_{L^\infty_M} \leq \|q\|_s\leq C_2(s)\|q\|_{L^\infty_M} \mbox{ for some  } C_1(s), C_2(s) \in \R^{*}_+,\] 
which yields \fi
Finally, we derive that $L^\infty_M(\mathbb{R})$ is also complete with  the norm $\|.\|_s$. 
\section{The existence assuming some technical results}\label{section-Proof-assuming-estimates}

As mentioned before, we only give the proof in the one dimensional case. This section is devoted to the proof of Theorem \ref{Theorem-principal}. We proceed in five steps, each of them making a separate subsection.
\begin{itemize}
\item In the first subsection, we define a shrinking set $V_{\delta,b_0}(s)$ and translate our goal of
making $q(s)$ go to $0$ in $L^\infty_M(\mathbb{R})$ in terms of belonging to $V_{\delta,b_0}(s)$.
\item In the second subsection We exhibit a $k$ parameter initial data family for equation \eqref{equation-q} whose coordinates are very small (with
respect to the requirements of $V_{\delta,b_0}(s)$) except for the $k+1$ first parameter $q_0,..,q_{2k-1}$.
\item In the third subsection, we solve the local in time Cauchy problem for equation \eqref{equation-q} coupled with some orthogonality condition.
\item In the fourth subsection,  using the spectral properties of equation \eqref{equation-q} , we reduce our
goal from the control of $q(s)$ (an infinite dimensional variable) in $V_{\delta,b_0}(s)$ to the control of
its $2k$ first components $(q_0,..,q_{2k-1})$ (a (k)-dimensional variable) in $\left[ -I(s)^{-\delta}, I(s)^{-\delta} \right]^{2k}$.
\item  In the last subsection, we solve the finite dimensional problem using the shooting lemma and conclude the proof of Theorem \ref{Theorem-principal}.
\end{itemize}

\subsection{Definition of the shrinking set $V_{\delta,b_0}(s)$}
In this part,  we introduce the shrinking set  that controls  the asymptotic  behaviors  of our solution 
\begin{definition}\label{definition-shrinking-set}
Let us consider  an integer $k > 1$, the reals $ \delta >0 $, $ b_0 >0$ and $M$ given by \eqref{defi-M}, we define    $V_{\delta,b_0}(s)$ be the set of all $(q,b) \in   L^\infty_M \times  \mathbb{R}$   satisfying 
\begin{equation}\label{bound-for-q-m}
    \left|  q_{m} \right| \le I^{-\delta }(s) ,\quad \forall\; 0\leq m \leq [M],\;\; m\not = 2k,
\end{equation}
\begin{eqnarray}
%\label{bound-q-m-ge-2k+1}
\left| q_{2k} \right| \le  I^{-2\delta } (s) , 
\end{eqnarray}
\begin{equation}\label{bound-for-q--}
    \left|  q_-   \right|_s \le  I^{-\delta}(s), \quad   
\end{equation}
and 
\begin{equation}\label{bound-b}
    \frac{b_0}{2}\leq b \leq 2 b_0,
\end{equation}
where  $q_m$ and $q_-$ defined in \eqref{decomposition-q2},  $I(s)$ defined as in \eqref{defi-I-s} and $|\cdot |_s$ norm defined in   \eqref{norm-q-||-s}.
\end{definition}

%\begin{rem}\label{remark-h}
%From  $V_{\delta, b_0 }(s)$'s definition in above, our result immediately follows the following fact 
%$$  (q,b)(s) \in  V_{\delta, b_0} (s), \forall s \ge s_0.  $$
%\end{rem}

\subsection{ Preparation of Initial data}

In this part, we aim to give a suitable family of  initial data for our problem. Let us consider $(d_0, d_1,...,d_{2k-1}) \in \R^{2k}$,$\delta>0 $ and $ b_0 >0$,  we  then define
\begin{equation}\label{initial-data-new}
   \psi(d_0,...,d_{2k-1},y,s_0)=\sum_{i=0}^{2k-1} d_i I^{-\delta }(s_0)  y^i,
\end{equation}
\iffalse
and $g$ will fixed later so that it guarantees the orthogonal  condition at $s_0$:
\begin{equation}\label{condition-g-initial-s-0}
    P_{2k}( q(s_0)) = 0, \text{ and }  \|g \|_{L^\infty_M} \le I^{-\delta}(s_0),
\end{equation}
\fi 
then, we have the following result
\begin{lemma}[Decomposition of initial data in different components]\label{lemma-initial-data}
 Let us consider    $(d_i)_{0\le i \le 2k-1} \in \R^{2k}$ satisfying $ \max_{0 \le i \le 2k-1 } \left| d_i \right| \le 1 $ and $b_0 >0$ given  arbitrarily.  Then, there exists $ \delta_1(b_0)$ such that for all $\delta \le \delta_1$, there exists $s_1(\delta_1, b_0) \ge 1$ such that for all $s_0 \ge s_1$,    the following properties are valid with $\psi(d_0,...,d_{2k-1})$ defined  in \eqref{initial-data-new}:
 \begin{itemize}
     \item[(i)]  There exits a quadrilateral $ \mathbb{D}_{s_0} \subset \left[-2,2\right]^{2k} $ such that the mapping   
    
     \begin{equation}\label{defi-mapping-Gamma-initial-data}
         \begin{gathered}
\Gamma: \mathbb{D}_{s_0} \to \mathbb{R}^{2k} \hfill \\
\hspace{-1.5cm} (d_0,...,d_{2k-1}) \mapsto (\psi_0,...,\psi_{2k-1}) \hfill \\ 
\end{gathered},
     \end{equation}
     
     is linear one to one from $ \mathbb{D}_{s_0}$ to $\hat{\mathcal{V}}(s_0)$, with   
     \begin{equation}\label{define-hat-V-A-s}
      \hat{\mathcal{V}}(s) = \left[ -I(s)^{-\delta}, I(s)^{-\delta} \right]^{2k},  
     \end{equation}
    
where $(\psi_0,...,\psi_{2k-1})$ are the coefficients of initial data $\psi(d_0,...,d_{2k-1})$ given by the decomposition \eqref{decomposition-q2}.
In addition to that, we have 
\begin{equation}\label{des-Gamma-boundary-ne-0}
    \left. \Gamma \right|_{\partial \mathbb{D}_{s_0}} \subset \partial \hat{\mathcal{V}}(s_0) \text{ and } \text{deg}\left(\left. \Gamma \right|_{\partial \mathbb{D}_{s_0}} \right) \ne 0.
\end{equation}
\item[(ii)] For all $(d_0,...,d_{2k-1}) \in \mathbb{D}_{s_0}$, the following estimates are valid
\begin{eqnarray}
\left| \psi_0  \right| \le I^{-\delta}(s_0),.... ,  \left| \psi_{2k-1}  \right| \le  I^{-\delta}(s_0), \quad 
\psi_{2k} =\psi_{M} =0 \text{ and } \psi_-\equiv 0.
\end{eqnarray}
 \end{itemize}
\end{lemma}
\begin{proof} 
The proof of the Lemma is quite the same as \cite[Proposition 4.5]{TZpre15}. 

- \textit{The proof of item (i):} 
From \eqref{initial-data-new}, definition's $H_n $ in \eqref{eigenfunction-Ls} and   \eqref{defi-q_m}, 
we get 
\begin{eqnarray}
\left| \psi_n (s_0) - d_n I^{-\delta}(s_0) \right| \le C(d_0,...,d_{2k-1}) I^{-\delta - 2}(s_0)
\end{eqnarray}
which concludes item (i). 

- \textit{ The proof of item (ii):}  From \eqref{initial-data-new}, $\psi(d_0,...,d_{2k-1},s_0)$ is a polynomial of order  $2k-1$, so it follows that 
$$ \psi_{n} =0, \forall n \in \{2k,..,M\} \text{ and }  \psi_{-} \equiv 0.$$
In addition, since  $(d_0,...,d_{2k-1}) \in \mathbb{D}_{s_0}$, we use item (i) to deduce that $(\psi_0,...,\psi_{2k-1}) \in \hat{\mathcal{V}}(s_0)$ and 
$$ \left| \psi_{n} \right| \le I^{-\delta}(s_0), \forall n \in \{ 0,...,2k-1\}  $$
which concludes item (ii)and the proof  Lemma \ref{lemma-initial-data}.
 \end{proof}

\begin{rem}  Note that   $s_0= -\ln (T)$ is the \textit{master constant}. In almost every argument in this paper it is almost to be sufficiently depending on the choice of all other constants ($\delta_0$ and $b_0$). In addition, we denote $C$ as the universal constant that is independent to $b_0$ and $s_0$. 
\end{rem}

\subsection{Local in time solution for the problem \eqref{equation-q}  $ \&$ \eqref{Modulation-condition}}

 As we setup in the beginning, besides  main solution $q$,  modulation flow $b$ plays an important role in our analysis, since it helps us to disable  the perturbation of the neutral mode corresponding to eigenvalue $\lambda_{2k} = 0$ of linear operator $\mathcal{L}_s$. In particular, the modulation flow is one of  the main  contributions of our paper. The uniqueness of the flow $b$ is defined  via the following  orthogonal condition
\beqtn\label{Modulation-condition}
 \langle q, H_{2k} \rangle_{L^2_{\rho_s}} = 0.
\eeqtn
As a matter of fact, the cancellation  ensures that $q_{2k} =0$, the projection of the solution on $ H_{2k}$, corresponding to eigenvalue $ \lambda_{2k} =0 $, since the neutral  issues   to the control of our solution. Consequently, our problem given by  \eqref{equation-q}  is coupled with the condition \eqref{Modulation-condition}. In the following, we  aim to establish  the local existence and uniqueness.
\begin{prop}[Local existence of the coupled  problem \eqref{equation-q} $\&$ \eqref{Modulation-condition}]  Let   $(d_{i})_{0\leq i\leq 2k-1} \in \mathbb{R}^{2k} $ satisfying $\max_{0\le i \le 2k-1} |d_i| \le 2$ and   $\delta >0, b_0 >0$,   there exists $s_2 ( \delta, b_0) \ge 1$, such that for all $ s_0 \ge s_2$, the following property holds:  If we  choose   initial data $\psi$ as in \eqref{initial-data-new},   then,  there exists $s^* > s_0$ such that    the coupled problem \eqref{equation-q} $ \&$ \eqref{Modulation-condition}   uniquely  has solution on $[s_0,  s^* ]$. Assume furthermore that  the solution   $(q,b)(s) \in \mathcal{V}_{ \delta, b_0}(s)$ for all $s \in [s_0,s^*]$, then, the solution  can be extended after the time $s^*$ i.e.  the existence and uniqueness of     $(q,b)$ are valid  on $[s_0,s^*+\varepsilon]$, for some $\varepsilon >0$ small.  
\end{prop}
\begin{proof}
  Let us consider  initial $w_0$ defined as in \eqref{decompose-equa-w-=q} with $q(s_0) = \psi(d_0,d_1,...,d_{2k-1}) $  given as in  \eqref{initial-data-new}, since equation \eqref{NLH} is  locally well-posedness in $L^\infty$, then, the solution $w$ to equation \eqref{equation-w} exists on $\left[s_0, \tilde s \right]$ for some $\tilde s > s_0 $. Next,  we need to prove that $w$ is uniquely decomposed  as in \eqref{decompose-equa-w-=q} and  $(q,b)(s) $ solves \eqref{equation-q} and  \eqref{Modulation-condition}. The result follows the  Implicit function theorem. Let us define  the functional $\mathscr{F}$  by 
\begin{equation}\label{defimathscr-F-functional}
\mathscr{F}(s,b) = \left\langle w f_b^{-p} - \left( p-1 +b y^{2k} \right), H_{2k}  \right\rangle_{L^2_{\rho_s}}.
\end{equation}
For $b_0 >0$,  and at $s=s_0$,  from $\psi(d_0,...,d_1)$'s definition in \eqref{initial-data-new}, it directly follows that   
\begin{equation}\label{equality-F=0}
    \mathscr{F}(s_0, b_0) =0.
\end{equation}
\iffalse
Regarding to \eqref{initial-data-new}, $g$ need to satisfy  
\begin{equation}\label{condition-on-g}
    \langle g f_{b_0}^{- p}, H_{2k} \rangle_{L^2_{\rho_s}} =0.
\end{equation}

\fi 

\noindent 
\medskip
Next, we aim to verify
\begin{eqnarray}\label{partial-F-s-0ne-0} 
\frac{ \partial \mathcal{F}}{\partial b } (s_0, b_0) \ne 0. 
 \end{eqnarray}
From \eqref{defimathscr-F-functional}, we obtain  
\begin{eqnarray}
\frac{\partial \mathcal{F}}{\partial b}(s,b) 
= \left\langle w \frac{p y^{2k}}{p-1} f_b^{-1} - y^{2k}, H_{2k}  \right\rangle_{L^2_{\rho_s}}.\label{formula-partial-F}
\end{eqnarray}

%%%%%%%%%%%%%%
\noindent 
According to \eqref{decompose-equa-w-=q}, we express  $w(s_0)$ as follows
$$ w(y,s_0) =   f_{b_0} \left( 1 + I^{-\delta}(s_0)f_{b_0}^{p-1}(y,s_0)\sum_{i=0}^{2k-1} d_i   y^i \right).  $$
Then, we have
\begin{eqnarray}
& & \frac{\partial \mathcal{F}}{\partial b}(s_0,b_0) =  I^{-\delta}(s_0)\frac{p}{p-1}\left\langle f_{b_0}^{p-1}(y,s_0) \sum_{i=0}^{2k-1} d_i  y^{i+2k}, H_{2k} \right\rangle_{L^2_{\rho_{s_0}}}
\label{scalar-product-modulation-1}\\
&+&\frac{1}{p-1} \left\langle   y^{2k}, H_{2k}  \right\rangle_{L^2_{\rho_{s_0}}}  :=A+B.\nonumber
\end{eqnarray}
Using \eqref{eigenfunction-Ls} and  \eqref{scalar-product-hm}, we immediately have  $$ B=\frac{2^{4k}(2k)!}{p-1}I^{-4k}(s_0).$$
\\
In addition, we use \eqref{defi-e-b} to get the following expression
 \begin{equation}\label{eb0-modulation}
     e_{b_0}(y) = (p-1)^{-1} \left( \sum_{l=0}^L \left(- \frac{b y^{2k}}{p-1} \right)^l + \left( -\frac{b y^{2k}}{p-1}\right)^{L+1} \right),
 \end{equation}
 for $L \in \mathbb{N}, L \ge 2$ arbitrarily.
 
 \noindent 
Now, we decompose the part $A$ in  \eqref{scalar-product-modulation-1} by
\[\begin{array}{l}
  A=I^{-\delta}(s_0)\frac{p}{p-1}\left\langle f_{b_0}^{p-1}(y,s_0) \sum_{i=0}^{2k-1} d_i   (s_0) y^{i+2k}, H_{2k} \right\rangle_{L^2_{\rho_{s_0}}}\\ 
  =\dsp I^{-\delta}(s_0)\sum_{i=0}^{2k-1}  d_i  \left (\int_{|y|\leq 1}e_{b_0}(y,s_0) y^{i+2k} H_{2k} \rho_{s_0}(y)dy+ \int_{|y|\geq 1}e_{b_0}(y,s_0) y^{i+2k} H_{2k} \rho_{s_0}(y)dy\right ) \\
  =A_1+A_2.\\
   \end{array}
\]
Since $e_{b_0}y^{2k}$ is bounded, we apply Lemma \ref{small-integral-y-ge-I-delta} to get
$$ |A_2| \lesssim I^{-4k-\delta}(s_0),$$
provided that $s_0 \ge s_{2,2}(\delta)$. Besides that, we use \eqref{eb0-modulation} with   $L\ge 2$  arbitrarily and we write $A_1$ as follows 
\begin{eqnarray*}
A_1 = (p-1)^{-1} I^{-\delta}(s_0) \sum_{i=0}^{2k-1} d_i \int_{|y| \le 1} \left[ \sum_{j=0}^L \left( -\frac{b \xi^{2k}}{p-1} \right)^l +\left( -\frac{b \xi^{2k}}{p-1} \right)^L  \right] y^{i +2k} H_{2k}(s_0) \rho_{s_0} dy. 
\end{eqnarray*}
Using Lemmas \ref{lemma-scalar-product-H-m} and \ref{small-integral-y-ge-I-delta}, we get 
\begin{equation*}
|A_1|\lesssim  I^{-4k - \delta}(s_0).
\end{equation*}
By adding all related terms, we obtain
\begin{equation*}
\frac{ \partial  \mathcal{F}}{ \partial b} (s_0, b_0) = I^{-4k}(s_0)2^{4k} (2k)! \left(1+    O(I^{-\delta}(s_0))\right) \ne 0,
\end{equation*}
provided that $s_0 \ge s_{2,3}(\delta, b_0)$. Thus, \eqref{partial-F-s-0ne-0} follows. 

By equality \eqref{equality-F=0} and \eqref{partial-F-s-0ne-0} and using the  Implicit function Theorem, we obtain the existence of a unique $s^* >0$ and  $b \in C^1(s_0,s^*)$ such that $q$ defined as in \eqref{decom-q-w-}, verifies \eqref{equation-q}, and the orthogonal condition \eqref{Modulation-condition} hold. Moreover,  if we assume furthermore that $(q,b)$ is shrunk  in the  set $V_{A,\delta,b_0}(s)$ for all 
$s \in [s_0,s^*]$, then, we can repeat  the computation for \eqref{expansion=partial-F-b-s-0} in using the bounds given in Definition  \eqref{definition-shrinking-set} and we obtain 
\begin{eqnarray*}
 \frac{ \partial \mathcal{F}}{ \partial b}  \left. \right|_{(s,b) = (s^*, b(s^*))}  = I^{-4k}(s^*) 2^{4k} (2k)! \left(  1 + O(I^{-\delta}(s^*)) \right) \ne 0.
\end{eqnarray*}
Then, we can apply the Implicit function theorem to get the existence and uniqueness of $(q,b)$ on the interval $[s^*,s^* +\varepsilon]$ for some $\varepsilon >0$ small and the conclusion of the Lemma completely follows. \end{proof} 

\subsection{Reduction to a finite dimensional problem}
As we defined shrinking set $V_{\delta, b_0 }$ in Definition \ref{definition-shrinking-set}, it is sufficient to prove there exists a unique global  solution $(q,b)$ on $[s_0, +\infty)$ for some $s_0 $ sufficient large  that 
$$ (q,b)(s) \in V_{\delta, b_0}(s), \forall s \ge s_0.$$
In particular, we show in this part that the control of  infinite problem is reduced to a finite dimensional one. To get this key result, we  first show the following \text{priori estimates }.
\begin{prop}[A priori estimates] 
\label{proposition-ode} Let $b_0 > 0$ and   $k \in \mathbb{N}, k \ge 2, b_0 >0$, then there exists $\delta_{3}(k, b_0) > 0$ such that  for all $\delta \in (0,\delta_3)$, there exists $s_3(\delta, b_0)$ such that for all $s_0 \ge s_3$,  the following property holds:  Assume  $(q,b)$ is a solution to problem \eqref{equation-q} $\&$ \eqref{Modulation-condition}  that  $(q,b)(s) \in \mathcal{V}_{\delta, b_0}(s) $ for all $s\in[\tau, \bar s]$ for some $\bar s \geq s_0$, and  $q_{2k}(s)=0$ for all $s\in [\tau, \bar s]$, then for all $s\in [\tau,s_1], s_0 \le \tau \le \bar s  $, the following properties hold:
\begin{itemize}
    \item[(i)] (ODEs on the finite modes). For all $j \in \{ 0,...,[M] \}$, we have
    $$\left |q_j'(s)-\left( 1-\frac{j}{2k}\right)q_j(s) \right |\leq CI^{-2\delta}(s). $$
    \item[$(ii)$] (Smallness of the modulation $b(s)$).  It satisfies that
     \begin{equation*}\label{estimat-derivative-b}
        \left| b'(s) \right| \leq C I^{-\delta}(s)\mbox{ and }\frac 34b_0\leq b(s)\leq \frac 54b_0.
    \end{equation*}
    \item[$(iii)$]  (Control of the infinite-dimensional part $q_-$):  We have the following a priory estimate
 \[
\begin{array}{lll}
\left| q_-(s)\right|_s &\le & e^{-\frac{s-\tau}{p-1}} \left| q_-(\tau)\right|_\tau +  C \left( I^{-\frac{3}{2} \delta}(s) + e^{-\frac{s-\tau}{p-1}}  I^{-\frac{3}{2}\delta}(\tau)\right).
\end{array}
\]
 \end{itemize}
\end{prop} 
  \begin{proof}[Proof of Proposition \ref{proposition-ode}]
 This result plays an important role in our proof. In addition to that, the proof based on a long computation which is technical. To help the reader  in  following the paper, we  will give the complete proof in Section   \ref{proposition-ode}. 
\end{proof}
Consequently, we have the following result
  \begin{prop}[Reduction to a finite dimensional problem]\label{propositionn-transversality}
  Let $b_0 >0$ and $ k \in \mathbb{N}, k \ge 2 $, then there exists $\delta_4(b_0)$ such that for all $ \delta \in (0,\delta_4)$, there exists $s_4(b_0, \delta)$ such that for all $s_0 \ge s_4$,  the following  property holds:  Assume that $(q,b)$ is a solution to \eqref{equation-q} $\&$ \eqref{Modulation-condition} corresponding to  initial data  $(q,b)(s_0) = (\psi(d_0,...,d_{2k-1}),s_0)$ where  $\psi(d_0,...,d_{2k-1}),s_0)$ defined as in \eqref{initial-data-new} with $ \max_{0 \le i \le 2k-1} |d_i| \le 2 $;   and $(q,b)(s)\in V_{\delta,b_0}(s)$ for all $s \in [s_0, \bar s]$ for some $\bar s > s_0$ that  $(q,b)( \bar s) \in \partial V_{\delta,b_0}( \bar s)$, then the following properties are valid:
\begin{itemize}
     \item[(i)] \textbf{(Reduction to finite modes)}: Consider $q_0,...,q_{2k-1}$ be projections defined as in   \eqref{defi-q_m} then, we have 
     $$\left (q_0,..,q_{2k-1}\right )(\bar s) \in \partial \hat{V}(\bar s),$$
     where $I(s)$ is given by \eqref{defi-I-s}.\\
    %\textcolor{blue}{ \item[(iii)] %(\textbf{Transversality}): There %exists $\nu_0 >$ such that 
  %   $$(q,b)(\bar s +\nu) \notin \mathcal{V}_{\delta, b_0}(\bar s +\nu), \forall \nu \in (0, \nu_0).$$
\item[(ii)]  \textbf{(Transverse crossing)} There exists $m\in\{0,..,2k-1\}$ and $\omega \in \{-1,1\}$ such that
       \[\omega q_m(s_1)=I(s_1)^{-\delta}\mbox{ and }\omega \frac{d q_m}{ds}>0.\]
 \end{itemize}
\end{prop}  
  
\begin{rem}
In (ii) of Proposition \ref{propositionn-transversality}, we show that the solution $q(s)$ crosses the boundary $\partial V_{\delta, b_0}(s)$ at $s_1$ with positive speed, in other words, that all points on $\pa V_{\delta, b_0}(s_1)$ are \textit{strict exit points} in the sense of \cite[Chapter 2]{Conbook78}.
\end{rem}
\begin{proof} Let us start the proof Proposition \ref{propositionn-transversality} assuming Proposition \ref{proposition-ode}. Let us consider $\delta \le \delta_3$ and $s_0 \ge s_3$ that Proposition \ref{proposition-ode} holds.\\

\noindent 
- \textit{Proof of item (i)}  To get the conclusion of this item, we aim to show that  for all $s \in [s_0, \bar s]$
\begin{equation}\label{improve-q-j-ge-2k+1}
    \left|   q_j(s)    \right| \le \frac{1}{2} I^{-\delta}(s), \forall j \in \{ 2k+1,...,[M] \} (\text{note that } q_{2k} \equiv 0),
\end{equation}
and 
\begin{equation}\label{improve-q_-}
    \left| q_-(s)  \right|_s \le  \frac{1}{2} I^{-\delta}(s), 
\end{equation}

+ For \eqref{improve-q-j-ge-2k+1}: From item (i) of Proposition  \ref{proposition-ode}, we have
$$  \left[ q_j(s)  \pm \frac{1}{2} I^{-\delta}(s)  \right]' =  \left( 1 - \frac{j}{2k} \right) q_j(s)  \pm \frac{\delta}{2} \left( \frac{1}{2k} - \frac{1}{2} \right) I^{-\delta}(s) + O(I^{-2 \delta}(s)).   $$
Hence, with $ j > 2k, \delta \le \delta_{4,1}$ and initial data $q_j(s_0) =0$ that $ q_j(s_0) \in \left( -\frac{1}{2}I^{-\delta}(s_0), \frac{1}{2} I^{-\delta}(s_0) \right) $, it follows that
$$  q_j(s) \in \left(  -\frac{1}{2}I^{-\delta}(s), \frac{1}{2} I^{-\delta}(s) \right), \forall s \in [s_0, \bar s_0],  $$
which concludes \eqref{improve-q-j-ge-2k+1}.

+ For \eqref{improve-q_-}: Let consider $\sigma \ge 1 $ fixed later.   We divide into two cases that $ s  - s_0 \le s_0 $ and $s - s_0 \ge s_0$. According to the first case, we apply item (iii) with $\tau = s_0$   that 
\begin{eqnarray*}
\left| q_-(s)  \right|_s \le C \left( I^{-\frac{3}{2}\delta}(s) + e^{-\frac{s-s_0}{p-1}} I^{-\frac{3}{2}\delta}(s_0) \right) \le \frac{1}{2} I^{-\delta}(s),
\end{eqnarray*}
provided that $\delta \le \delta_{4,2}$ and   $s_0 \ge s_{4,2}(\delta) $. In the second case, we use item (iii) again  with $\tau = s - s_0 $, and we obtain
\begin{eqnarray*}
\left| q_-(s)  \right|_s & \le &  e^{-\frac{s_0}{p-1}} I^{-\delta}(\tau)  + C\left( I^{-\frac{3}{2}\delta} + e^{-\frac{s_0}{p-1}} I^{-\frac{3}{2}\delta}(\tau)   \right)\\
&  \le  &  C  ( e^{-\frac{s_0}{p-1}} I^{\delta}(s) I^{-\frac{3}{2}\delta}(\tau) + I^{-\frac{1}{2}\delta}(s)  )  I^{-\delta}(s) \le \frac{1}{2} I^{-\delta}(s).
\end{eqnarray*}
Thus, \eqref{improve-q_-} follows. Finally, using  the definition of $V_{\delta, b_0}(s)$, the fact $(q,b)(\bar s) \in \partial  V_{\delta, b_0}(\bar s)$, estimates \eqref{improve-q-j-ge-2k+1}, \eqref{improve-q_-}, and item (ii) of Proposition \ref{propositionn-transversality}, we get the conclusion of  item (ii).
\iffalse
We define $\sigma=s-\tau$ and take $s_0=-\ln T\ge 3 \sigma,\;(\mbox{ that is } T\le e^{-3\sigma})$ so that for all $\tau \geq s_0$ and $s\in [\tau, \tau +\sigma]$, we have

\[\tau\leq s\leq \tau +\sigma\leq \tau+\frac{1}{3} s_0\leq \frac{4}{3} \tau.\]

From (i) of proposition Proposition \ref{proposition-ode}, we can write for all $2k\leq j\leq [M] $
  
  \[\left |\left ( e^{-(1-\frac{j}{2k})t}q_j(t)\right )' \right |\leq  C e^{-(1-\frac{j}{2k})t}I^{-2\delta}(t),\]
  We consider that $\tau \leq s\leq \frac{4}{3}\tau$.   Integrating the last inequality between $\tau$ and $s$, we obtain
%\[|\int_{\tau}^{s}
\[
\begin{array}{lll}
|q_j(s)|&\leq &e^{-(1-\frac{j}{2k})(\tau-s)}  q_j(\tau)+ C(s-\tau)e^{(1-\frac{j}{2k})s} I^{-2\delta}(\tau)\\
&\leq &e^{(1-\frac{j}{2k})(s-\tau)}  q_j(\tau)+ C(s-\tau)e^{(1-\frac{j}{2k})s} I^{-\frac{4}{3}\delta}(s),
\end{array}
\]

There exists $\tilde s_1=\max \{s_j,\;5\le j\le 9\}$, such that if $s\geq \tilde s_1$, then we can easily derive
\[|q(s)|\leq Ce^{(1-\frac{j}{2k})(s-\tau)} I^{-\frac{4}{3}\delta}(s)+I^{-\frac{7}{6}\delta}(s)< \frac{1}{2}I^{-\delta}(s).\]
  In a similar fashion, exists $\tilde s_2=\max \{s_j,\;9\le j\le 13\}$, for all $s\geq \tilde s_2$, we obtain 
  \[|q_-(s)|_s< \frac{1}{2}I^{-\delta}(s).\]
  %For the $q_2k$, we have by (ii) of Proposition \ref{proposition-ode}
  Thus we finish the prove of (ii) of Proposition \ref{propositionn-transversality}.\\
  \fi
  
 \noindent  
- \textit{Proof of item (ii)}: From item  (ii) of Proposition \ref{propositionn-transversality}, there exist $m=0,..2k-1$ and $\omega=\pm 1$ such that $q_m(s_1)=\omega I(s_1)^{-\delta}$. By (ii) of Proposition \ref{proposition-ode}, we see that for $\delta>0$
  \[\omega q_m'(s_1)\geq (1-\frac{m}{2k})\omega q_m(s_1)-CI^{-2\delta}(s_1)\geq C\left ((1-\frac{m}{2k})I^{-\delta}(s_1)-I^{-2\delta}(s_1)\right )>0,\]
  which concludes the proof of Proposition \ref{propositionn-transversality}. It remains to prove Proposition \ref{proposition-ode}. This will be done in Section \ref{Section-proof-proposition-ode}.
  \end{proof} 
 \subsection{Topological\textit{ ``shooting method``} for the finite dimension problem and proof of Theorem \ref{Theorem-principal}}

In this part we aim to give the complete proof to Theorem  \ref{Theorem-principal}  by using a topological \textit{shooting method}:
\begin{proof}[The proof of Theorem \ref{Theorem-principal}]
   Let us consider  $\delta>0$, $T > 0,(T= e^{-s_0})$ , $(d_0,..,d_{2k-1}) \in \mathbb{D}_{s_0}$ such that problem \eqref{equation-q} $\&$
\eqref{Modulation-condition} with initial data  
$ \psi(d_0,...,d_{2k-1},s_0)$
defined as in  \eqref{initial-data-new} has a  solution $(q(s),b(s))_{d_0,..,d_{2k-1}}$ defined for all $s \in  [s_0,\infty)$ such that
\beqtn
\|q(s)\|_{L^\infty_M} \leq C I^{-\delta}(s)\mbox{ and } |b(s)-b^*|\leq C I^{-2\delta}(s),
\label{goal-of-the proof}
\eeqtn
for some $b^*>0$.

 Let $b_0,\delta$ and $ s_0$  such that  Lemma \ref{initial-data-new}, Propositions \ref{propositionn-transversality} and  Proposition \ref{proposition-ode} hold,  and we denote $T= e^{-s_0}$ (positive since $s_0$ is large enough).
We proceed by contradiction, from (ii) of Lemma \ref{initial-data-new}, we assume that for all $(d_0,...,d_{2k-1}) \in \mathbb{D}_{s_0}$ there exists $s_*=s_*(d_0,..,d_{2k-1}) < +\infty$ such that 
\begin{equation*}
\begin{array}{ll}
     q_{d_0,..,d_{2k-1}}(s)\in V_{\delta,b_0}(s), & \forall s\in [s_0, s_*],  \\
     q_{d_0,..,d_{2k-1}}(s_*)\in \pa  V_{\delta,b_0}(s_*).&
\end{array}
\end{equation*}
By using item  (i) of Proposition \ref{propositionn-transversality}, we get $(q_0,..,q_{2k-1})(s_*)  \in \pa \hat{\mathcal{V}}(s_*)$ and we introduce $\Phi$  by
\[\Phi:
\begin{array}{ll}
\mathbb{D}_{s_0}\to \pa [-1,1]^{2k}&\\
(d_0,..d_{2k-1})\to I^{\delta}(s)(q_0,..,q_{2k-1})(s_*),
\end{array}
\]
which is well defined and  satisfies the following properties:
\begin{itemize}
\item[$(i)$] $\Phi$ is continuous from  $\mathbb{D}_{s_0}$ to  $\pa [-1,1]^{2k}$ thanks to the continuity  in time of $q$  on the one hand, and the continuity of $s_*$ in $(d_0,...,d_{2k-1})$ on the other hand,  which  is a direct consequence  of the trasversality in item (ii) of Proposition \ref{propositionn-transversality}.
\item[(ii)] It holds that $\Phi \left. \right|_{\partial \mathbb{D}_{s_0}}$ has nonzero degree. Indeed, for all $(d_0,...,d_{2k-1})  \in \partial \mathbb{D}_{s_0}$, we derive from item (i) of Lemma \ref{lemma-initial-data}  that $s_*(d_0,...,d_{2k-1})  =s_0$ and 
$$ \text{ deg}\left( \Phi \left. \right|_{\partial \mathbb{D}_{s_0}} \right) \ne 0.  $$
\end{itemize}
From Wazewski's principle in degree theory such a $\Phi$ cannot exist. Thus, we can prove that there exists $(d_0,...,d_{2k-1}) \in \mathbb{D}_{s_0}$ such that the corresponding solution $(q,b)(s) \in  V_{\delta, b_0}(s), \forall s \ge s_0$. 
\iffalse

 and by (iii) of Proposition \ref{propositionn-transversality}, $\Phi$ is continuous.\\
In the following we will prove that $\Phi$ has nonzero degree, which mean by the degree theory (Wazewski's principle) that for all $s\in [s_0, \infty )$ $q(s)$ remains in $V_{\delta,b_0}(s)$, which is a contradiction with the Exit Proposition.\\
Indeed Using Lemma \ref{initial-data-new}, and the fact that $q(-\ln T)=\psi_{d_0,..,d_{2k-1}}$, we see that when $(d_0,..,d_{2k-1})$ is on the boundary of the quadrilateral $\mathbb{D}_T$, $q_0,..,q_{2k-1}(-\ln T)\in \pa [-I^{-2\delta}(s),I^{-2\delta}(s)]^{2k}$ and $q(-\ln T)\in V_{\delta,b_0}(-\ln T)$ with strict inequalities for the other components.\\
By the Exit proposition \ref{propositionn-transversality}, $q(s)$ leaves $V_{\delta,b_0}$ at $s_0=-\ln T$, hence $s_*=-\ln T$.\\ 
Using (ii) of Proposition \ref{propositionn-transversality}, we get that the restriction of $\Phi$ on he boundary of $\mathbb{D}_T$ is of degree $1$, which means by the shooting method that for all $s\in [s_0, \infty )$ $q(s)$ remains in $V_{\delta,b_0}(s)$, which is a contradiction.\\
We conclude that there exists a value $(d_0,..,d_{2k-1})\in \mathbb{D}$ such that for all $s\geq -\ln T$, $q_{d_0,..,d_{2k-1}}\in V_{\delta,b_0}(s)$, which means that 
\beqtn\label{estimation-linftyM-q}
\left \|\frac{q}{1+|y|^M}\right\|_{L^\infty}\leq C I^{-\delta}(s),
\eeqtn
and using the  definition of\fi 
In particular, we derive from  \eqref{decompose-equa-w-=q},  $M=\frac{2kp}{p-1}$, and the following estimate  
\[|f_be_b|=|f_b^p|\leq C(1+|y|^{-\frac{2kp}{p-1}})= C(1+|y|^{-M}) \]
that 
\[\|w(y,s)-f_{b}\|_{L^\infty}=\|f_{b}e_bq\|_{L^\infty} \leq C I^{-\delta}(s).\]
%where  we use the notation $b(t) = b(s)$. 
So, we conclude item  (i) of Theorem \ref{Theorem-principal}.\\
The proof of item (ii):  From (ii) of Proposition \ref{proposition-ode}, it immediately follows that there exists $b^* \in \mathbb{R}^*_+$ such that 
$$ b(s) \to b^* \text{ as } s \to +\infty,  $$
which is equivalent to 
$$ b(t) \to b^* \text{ as } t \to T.$$
In particular, by integrating the first inequality given by  between $s$ and $\infty$ and using the fact that $b(s)\to b^*$ (see \eqref{convegence-b-s}), we obtain
\[|b(s)-b^*|\leq Ce^{-\delta s(1-\frac{1}{k})}.\]
Note that $s = -\ln(T-t)$ then,  \eqref{goal-of-the proof}  follows and the conclusion of item (ii) of Theorem \ref{Theorem-principal}. 
\\
%\noindent
%Finally, we totally conclude the proof of 
 %Theorem \ref{Theorem-principal}.
\end{proof}
\section{Proof to  Proposition \ref{proposition-ode}  }\label{Section-proof-proposition-ode}

In this section, we prove Proposition \ref{proposition-ode}. We just have to project equation \eqref{equation-q}
to get equations satisfied by the different coordinates of the decomposition \eqref{decomposition-q2}. More
precisely, the proof will be carried out in 2 subsections,

\begin{itemize}
\item In the first subsection, we write equations satisfied by $qj$, $0\le j\leq M$, then, we prove (i), (ii)  of Proposition \ref{proposition-ode}.
\item In the second subsection, we first derive from equation \eqref{equation-q} an equation satisfied by $q_-$ and prove the last identity in (iii) of Proposition \ref{proposition-ode}.

\end{itemize}

\subsection{The proof to items (i) and (ii) of Proposition \ref{proposition-ode} }\label{subsection-proof-i-ii}
\begin{itemize}
\item In Part 1, we project equation \eqref{equation-q} to get equations satisfied by $q_j$  for $0 \leq j\leq [M]$.
%and we then estimate its derivatives.
\item In Part 2: We will use  the precise estimates from  part I to conclude items (i) and  (ii) of Proposition \ref{proposition-ode}.
\end{itemize}

\medskip

\textbf{Part 1: The projection of equation \eqref{equation-q} on the eigenfunctions of the operator $\mathcal{L}_s$.}

Let  $(q,b)$  be  solution to problem \eqref{equation-q} $\&$ \eqref{Modulation-condition} trapped in $V_{\delta, b_0}(s)$ for all $s \in [s_0, \bar s]$ for some $\bar s > s_0$. Then, we have the following:

\medskip

\textbf{a) First term $\pa_s q$:}  In this part, we aim to estimate the error between $\partial_s q_n(s)$ and $P_n(\partial_s q)$ by the following Lemma 
\iffalse
\begin{lemma}
There exist $\delta_0 > 0 $ such that for all $0<\delta < \delta_0$ and $b_0 >0$, there exists $s_{2}(A,\delta, b_0) \ge 1$ such that for all $ s_0 \ge s_2$, the following  property holds: Assume that $(q,b) \in V_{\delta, b_0}(s )$ for all $s \in [s_0,s^*],  $ for some $s^* >s_0$, satisfying    \eqref{equation-q}- \eqref{modulation-equation},  then, the following estimates  hold
 \begin{equation}\label{esti-par-q-m-P-partial-s-q}
     \left| \partial_s q_n -  P_n(\partial_s q)  -n \left( 1- \frac{1}{k} \right)  q_n    \right| 
     \le C A^{\max(n-1,0)} I^{j -2k } (s), \quad \forall s \in (s_0,s^*), \forall n \in \{0,1,...,2k-1\}.
 \end{equation}
and for $ n=2k$, we have 
\begin{equation}\label{partial-q-2k-s}
\left| \partial_s q_{2k} -  P_{2k} (\partial_s q) - 2k \left( 1-\frac{1}{k} \right) q_{2k}   \right|   \le CA I^{-\delta}(s).   
\end{equation}
\end{lemma}
\fi 
\begin{lemma}\label{Lemma-Pn_partialq} For all $n \in \{0,1,...,[M]\}$, it holds that 
$$  P_{n} (\partial_s q)=\partial_s q_n (s) -  \left (1-\frac 1k\right )(n+1)(n+2) I^{-2}(s) q_{n+2} (s), \forall s \in [s_0, \bar s].  $$
\end{lemma}

%
%\begin{corollary} There exist a $\delta_0>0$ such that, for all $0<\delta \leq \delta_0$, $q$ in the shrinking set given by Definition \ref{definition-shrinking-set}, there exist $s_1(\delta)$, such that for all $s\geq s_1$
%\beqtn
%\left | P_n(\pa_s q)-\pa_s q_n \right |\leq C I^{-2\delta}(s).
%\eeqtn
%\end{corollary}

\begin{proof}We only give the proof when $n\geq 2$, for $n=0,1$ it is easy to derive the result. Using   \eqref{defi-q_m},  we have the following equality 
$$\langle H_n, H_n \rangle_{L^2_{\rho_s}}  q_n(s)  = \langle  q,H_n(s)\rangle_{L^2_{\rho_s}},$$
which implies
\begin{eqnarray*}
\langle H_n, H_n \rangle_{L^2_{\rho_s}} \partial_s q_n(s) & = &\langle \partial_s q, H_n \rangle_{L^2_{\rho_s}} + \langle  q, \partial_s H_n (s)\rangle_{L^2_{\rho_s}} + \left\langle  q,  H_n (s) \frac{ \partial_s \rho_s}{\rho_s} \right\rangle_{L^2_{\rho_s}} \\
& -  & \partial_s \langle H_n,H_n \rangle_{\rho_s} q_n,
\end{eqnarray*}
which yields

\begin{eqnarray*}
P_n(\pa_s q)  & = &  \partial_s q_n  - \langle  q, \partial_s H_n (s)\rangle_{L^2_{\rho_s}}\langle H_n, H_n \rangle_{L^2_{\rho_s}} ^{-1} \\
&   &  - \left\langle  q,  H_n (s) \frac{ \partial_s \rho_s}{\rho_s} \right\rangle_{L^2_{\rho_s}} \langle H_n, H_n \rangle_{L^2_{\rho_s}} ^{-1}  +  \partial_s \langle H_n,H_n \rangle_{\rho_s}  \langle H_n, H_n \rangle_{L^2_{\rho_s}} ^{-1} q_n.
\end{eqnarray*}

\medskip

\begin{eqnarray*}
\partial_s q_n  & = &\langle \partial_s q, H_n \rangle_{L^2_{\rho_s}}\langle H_n, H_n \rangle_{L^2_{\rho_s}} ^{-1} + \langle  q, \partial_s H_n (s)\rangle_{L^2_{\rho_s}}\langle H_n, H_n \rangle_{L^2_{\rho_s}} ^{-1} \\
&   &  + \left\langle  q,  H_n (s) \frac{ \partial_s \rho_s}{\rho_s} \right\rangle_{L^2_{\rho_s}} \langle H_n, H_n \rangle_{L^2_{\rho_s}} ^{-1}  -  \partial_s \langle H_n,H_n \rangle_{\rho_s}  \langle H_n, H_n \rangle_{L^2_{\rho_s}} ^{-1} q_n.
\end{eqnarray*}
Thus, we can write
\begin{equation}\label{estimate-q-s-partial-sn}
    \partial_s q_n =   P_{n}( \partial_s q) + \tilde  L,
\end{equation}
where 
\begin{eqnarray*}
\tilde L  =  \langle  q, \partial_s H_n (s)\rangle_{L^2_{\rho_s}}\langle H_n, H_n \rangle_{L^2_{\rho_s}} ^{-1} 
  + \left\langle  q,  H_n (s) \frac{ \partial_s \rho_s}{\rho_s} \right\rangle_{L^2_{\rho_s}} \langle H_n, H_n \rangle_{L^2_{\rho_s}} ^{-1}  -  \partial_s \langle H_n,H_n \rangle_{\rho_s}  \langle H_n, H_n \rangle_{L^2_{\rho_s}} ^{-1} q_n.
\end{eqnarray*}
We now aim to estimate $\tilde L$ provided that $(q(s),b(s)) \in V_{A,b_0,\delta}(s) $ and we   also  recall that 
$$ q = \sum_{j=1}^M q_j H_j + q_-. $$

+ For $\partial_s \langle H_n,H_n \rangle_{\rho_s}  \langle H_n, H_n \rangle_{L^2_{\rho_s}} ^{-1} q_n $: We have the facts that 
$$\langle H_n,H_n \rangle_{\rho_s} = I^{-2n}(s) 2^n n!, \text{ and  } I(s) = e^{\frac{s}{2}\left(1 -\frac{1}{k} \right) } ,     $$
which implies 
\begin{eqnarray*}
\partial_s  \langle H_n, H_n \rangle_{L^2_{\rho_s}} =  - n \left( 1 -\frac{1}{k}  \right) \langle H_n, H_n \rangle_{L^2_{\rho_s}}.
\end{eqnarray*}
So, we obtain 
$$ \partial_s \langle H_n,H_n \rangle_{L^2_{\rho_s}}  \langle H_n, H_n \rangle_{L^2_{\rho_s}} ^{-1} q_n(s)  = - n \left( 1 -\frac{1}{k} \right) q_n(s).$$

+ For $\left\langle  q,  H_n (s) \frac{ \partial_s \rho_s}{\rho_s} \right\rangle_{L^2_{\rho_s}} \langle H_n, H_n \rangle_{L^2_{\rho_s}}^{-1} $: Using the fact that 
$$ \partial_s \rho_s = \frac{1}{2} \left( 1 -\frac{1}{k}\right) \rho_s  - \frac{1}{4} \left( 1 - \frac{1}{k}  \right) I^2(s) y^2 \rho_s ,  $$
which yields
\begin{eqnarray*}
\left\langle  q,  H_n (s) \frac{ \partial_s \rho_s}{\rho_s} \right\rangle_{L^2_{\rho_s}} & = & \frac{1}{2} \left( 1 -\frac{1}{k} \right) \langle q, H_n(s) \rangle_{L^2_{\rho_s}} - \frac{1}{4} \left( 1 -\frac{1}{k} \right) \langle q, I^2(s) y^2 H_n(s) \rangle_{L^2_{\rho_s}}\\
&  =  & \frac{1}{2} \left( 1 -\frac{1}{k} \right) q_n  \langle H_n,H_n \rangle_{L^2_{\rho_s}}  - \frac{1}{4}\left( 1 -\frac{1}{k} \right) \langle q, I^2(s) y^2 H_n(s) \rangle_{L^2_{\rho_s}} .
\end{eqnarray*}
Thus, we derive 
\begin{eqnarray*}
\left\langle  q,  H_n (s) \frac{ \partial_s \rho_s}{\rho_s} \right\rangle_{L^2_{\rho_s}} \langle H_n, H_n \rangle_{L^2_{\rho_s}}^{-1} = \frac{1}{2} \left(1 -\frac{1}{k} \right) q_n - \frac{1}{4} \left(1 -\frac{1}{k} \right) \langle q, I^2(s) y^2 H_n(s) \rangle_{L^2_{\rho_s}} \langle H_n, H_n \rangle_{L^2_{\rho_s}}^{-1}
\end{eqnarray*}
%By the facts that
%\begin{eqnarray*}
%y^2 H_n(s) =       H_{n+2}(s) + \left( 4n+2\right) I^{-2}(s) H_{n}(s) + \sum_{j=1}^{\left[\frac{n}{2} \right]} a_j I^{-2j-2}(s) H_{n-2j}(s), \text{ for some } a_j \in \mathbb{R} ,
%\end{eqnarray*}
%\begin{equation}
%y^2 H_n(y,s)  = H_{n+2} + (4n+2) I^{-2}(s) H_n(y,s)  + c_{2,n} I^{-4}(s) H_{n-2}  (y,s),        
%\end{equation}
%for some $c_{2,n} \in \mathbb{R} $ and  $c_{2,n} = 0  $ if $ n =0,1$. Since 
%$$  \int_{\mathbb{R} } y^2 H_n(y,s)  . H_{j}(y,s) \rho_s(y) dy  =0, \forall j \le n-3.   $$
%In particular, we can specify the constant $ c_{2,n}$ by 
%\begin{eqnarray*}
 %\frac{\displaystyle\int_{\mathbb{R}} y^2 H_n (y,s) H_{n-2} (y,s) \rho_s(y) dy }{ \| H_{n-2}\|^2_{L^2_{\rho_s}}} = 4(n)(n-1) I^{-4}(s) = I^{-4}(s) c_{2,n}.
%\end{eqnarray*}
%Hence, we get
%$$ c_{2,n} = 4 n(n-1). $$
%Finally, we find the following identify 

 Using the polynomial Hermite identities, we obtain
   \[z^2h_n=zh_{n+1}+2nzh_{n-1}=h_{n+2}+2(2n+1)h_n+4n(n-1)h_{n-2},\]

and we find the following identify 
\begin{equation*}
y^2 H_{n}(y,s) =    H_{n+2} (y,s) + (4n+2) I^{-2}(s) H_n(y,s) + 4 n(n-1) I^{-4}(s) H_{n-2}(y,s)  
\end{equation*}
\iffalse

note that when $n=0,$ or $1$ the sum in the right hand side vanishes, and  $q \in V_{A,\delta, b_0}(s), \forall s \in [s_0,s^*]$,  we get
\begin{eqnarray*}
\langle q, I^2(s) y^2 H_n(s) \rangle_{L^2_{\rho_s}} \langle H_n, H_n \rangle_{L^2_{\rho_s}}^{-1} = (4n+2) q_n + O(AI^{-\delta -2}(s)).
\end{eqnarray*}
\fi
This implies that
\begin{eqnarray*}
& & \langle q, I^2(s) y^2 H_n(s) \rangle_{L^2_{\rho_s}} \\
& = & I^{2}(s)  \left[ q_{n+2} \| H_{n+2}\|^2_{L^2_{\rho}}  + I^{-2}(s) q_n (s) \|H_n\|^2_{L^2_\rho} + 4n(n-1)  q_{n-2} I^{-4}(s) \|H_{n-2}\|^2_{L^2_{\rho}} \right],
\end{eqnarray*}
which yields
\begin{eqnarray*}
& & \left\langle  q,  H_n (s) \frac{ \partial_s \rho_s}{\rho_s} \right\rangle_{L^2_{\rho_s}} \langle H_n, H_n \rangle_{L^2_{\rho_s}}^{-1} \\
&= &
-n \left(1 - \frac{1}{k} \right)q_n -  n(n-1) \left(1-\frac{1}{k} \right) q_{n-2} I^{-2}(s) \frac{\|H_{n-2}\|^2_{L^2_\rho}}{\|H_{n}\|^2_{L^2_\rho}}   \\
&+&\left(1-\frac{1}{k} \right) (n+2)(n+1)   I^{-2}(s) q_{n+2}, 
\end{eqnarray*}
for all $  n \in \{0,...,[M]\} \text{ and } \forall s \in [s_0,s^*]$ (with convention that $q_j =0$ if $j<0$) and for some $\tilde c_n \in \mathbb{R}$. \\
   + $ \langle  q, \partial_s H_n (s)\rangle_{L^2_{\rho_s}}\langle H_n, H_n \rangle_{L^2_{\rho_s}} ^{-1} $:  
   %Recall the identify on Hermite's  polynomial as follows 
%\[h_{n+1}=yh_n-nh_{n-1},\]
% which implies from \eqref{eigenfunction-Ls} that
% \textcolor{red}{$$H_{n+1}=yI^{-1}H_n-2nI^{-2} H_{n-1} \text{ and }  2nyH_{n-1}=Iy^2 %H_n-I^{2}yH_{n+1}. $$
% $$nyH_{n-1}=nIH_{n}+2n(n-1)I^{2}H_{n-2}$$
 %}

%$$H_{n+1}=IyH_n-nI^2 H_{n-1} \text{ and }  nyH_{n-1}=y^2 I^{-1}H_n-I^{-2}yH_{n+1}. $$
%Then,  we obtain 
\begin{eqnarray*}
\partial_s H_n(s) &=& - n I'(s) I^{-n-1}(s) h_n(I(s)y) +  I'(s) y h'_n(I(s)y) I^{-n}(s)   \\
& = & - \frac{n}{2} \left(1-\frac{1}{k} \right) H_n(s) + \frac{n}{2} \left( 1-\frac{1}{k} \right)  y H_{n-1}(s).
\end{eqnarray*}
Let us recall the following identify on Hermite's polynomial
\begin{equation}
    y H_{n-1} (y,s) = H_n(y,s) + I^{-2} (s) 2(n-1) H_{n-2}(y,s).
\end{equation}
So, we can rewrite $\partial_s H_n$ as follows
\begin{equation}\label{formula-partia-s-Hn}
   \partial_s H_n (y,s) = n (n-1) \left(1 -\frac{1}{k} \right) I^{-2}(s)  H_{n-2}(y,s).
\end{equation}
Thus, we obtain
\begin{eqnarray*}
\langle  q, \partial_s H_n (s)\rangle_{L^2_{\rho_s}}\langle H_n, H_n \rangle_{L^2_{\rho_s}} ^{-1}  &=& n (n-1) \left(1-\frac{1}{k} \right) I^{-2}(s)q_{n-2} \frac{\|H_{n-2}\|^2_{L^2_{\rho}}}{ \| H_{n}\|^2_{L^2_\rho}}.
\end{eqnarray*}
Finally, we obtain
\begin{equation*}
 \partial_s q_n = P_{n}( \partial_s q) + \left (1-\frac 1k\right )(n+1)(n+2) q_{n+2}, \forall n \in \{0,1....,[M]\},
\end{equation*}
which concludes the proof of the Lemma. 
\end{proof}

\textbf{ b) Second term $\mathcal{L}_s (q)$} 

\begin{lemma}\label{Lemma-P-n-mathcal-L-s} For all $0\leq n\leq [M]$, it  holds that 
\begin{equation}\label{P-n-mathcal-L-n-s-q}
    P_n( \mathcal{L}_s q)= \left(1-\frac{n}{2k}\right)q_n+(1-\frac{1}{k})(n+1)(n+2) I^{-2} q_{n+2}.
\end{equation}
\end{lemma}
\begin{proof} As in the proof of Lemma \ref{Lemma-Pn_partialq}, we only give the proof when $n\geq 2$, for $n=0,1$ it is easy to derive the result. We write $P_n( \mathcal{L}_s q)$ as follows:
\[\begin{array}{lll}
      P_n(\mathcal{L}_s q)&=& \dsp \int\left ( I^{-2}(s) \Delta q-\frac{1}{2}y \cdot \nabla q+q \right )H_n \rho_s dy+ \int\frac 12(1-\frac{1}{k})y\nabla q H_n \rho_s dy  \\
     & =&A_1+\frac 12(1-\frac{1}{k})A_2.
\end{array}
\]

In the following we will use Hermite polynomial identity \eqref{Hermite-identities-ell=2} given bu Lemma \ref{Hermite_Identies}. Using integration by part and polynomial identities we obtain

%dsp \langle  \mathcal{L}_s q, H_n \rangle_{L^2_{\rho_s}}&=&
\[
\begin{array}{rcl}
 \dsp A_1 &=& \int \left (I^{-2}(s) \Delta q-\frac{1}{2}y \cdot \nabla q+q\right ) H_n \rho_s dy\\
&=&\dsp  \int I^{-2}\div{(\nabla q\rho_s)} +q H_{n} dy,\\
&=& -\dsp I^{-2} \int\nabla q  nH_{n-1}\rho_s dy+q_n\|H_n\|_{L^2_{\rho_s}}^2,\\
&=&\dsp n(n-1)I^{-2}\int q H_{n-2} \rho_s-\frac n 2 \int y   qH_{n-1}\rho_s dy+q_n\|H_n\|_{L^2_{\rho_s}}^2\\
&=&-\dsp \left (1-\frac{n}{2}\right ) q_n \|H_n\|_{L^2_{\rho_s}}^2.

%\|H_n\|_{\rho_s}.
\end{array}
\]
By a similar computation, using the change of variable $z=Iy$ and we introduce $\rho(z)=I^{-1}\rho_s(y)$, we get 
\[\begin{array}{l}
   A_2=\dsp \int  y\nabla q H_n\rho_s dy\\
   = \dsp \left (- \int q H_{n}\rho_s dy- \int q nyH_{n-1}\rho_s dy + \frac{I^2}{2}\int qy^2H_{n} \rho_sdy\right ) \\
   = \dsp \left (-q_n\|H_n\|^2-I^{-n} n\int q z h_{n-1}\rho dz + \frac{1}{2} I^{-n}\int z^2 q h_{n}\rho dz\right ) \\
 \dsp \left (-q_n\|H_n\|^2-I^{-n} n\int q (h_n+(n-1)h_{n-2})\rho dz + \frac{1}{2} I^{-n}\int z^2 q h_{n}\rho dz\right ) \\
  \end{array}
   \]
   Using the polynomial Hermite identities that 
   \[z^2h_n=zh_{n+1}+2nzh_{n-1}=h_{n+2}+2(2n+1)h_n+4n(n-1)h_{n-2},\]
 which yields
 \[\begin{array}{rcl}
   A_2 &=& \dsp \left(-q_n\|H_n\|^2 -I^{-n} n\int q z h_{n-1} \rho dz + \frac{1}{2} I^{-n-2}\int z^2 q h_{n}\rho dz\right) \\
 &=&\dsp -q_n\|H_n\|^2-I^{-n} n\int q (h_n+2(n-1)h_{n-2})\rho dz \\
 &+& \frac{1}{2} I^{-n}\int q [h_{n+2}+2(2n+1)h_n+4n(n-1)h_{n-2}]\rho dz \\
 & = & -q_n\|H_n\|^2-nq_n\|H_n\|^2-2n(n-1) I^{-2}q_{n-2}\|H_{n-2}\|^2+\frac{1}{2}q_{n+2}I^2\|H_{n+2}\|^2\\
  &+& (2n+1) q_n\|H_n\|^2+2n(n-1)q_{n-2}I^{-2}\|H_{n-2}\|^2\\
  &=&\left (n q_n+ 2(n+2)(n+1) I^{-2} q_{n+2}\right )\|H_n\|^2.
  \end{array}
   \] 
 Thus, we obtain  by adding all related terms that 
\[\dsp P_n( \mathcal{L}_s q)= \left(1-\frac{n}{2k}\right)q_n+(1-\frac{1}{k})(n+2)(n+1) I^{-2} q_{n+2},
\]
which concludes the proof of Lemma \ref{Lemma-P-n-mathcal-L-s}. 
\end{proof}

\textbf{c) Third term, the nonlinear term $\mathcal{N} (q)$}\\ 

In this part, we aim to estimate to the projection of $\mathcal{N}(q)$ on $H_n,$ for some $n \in \{0,1,...,[M] \} $. More precisely, we have the following Lemma: 
\begin{lemma}\label{projecion-H-n-N}  Let $b_0 > 0$,  then, there exists $ \delta_5(b_0)>0$ such that for all  $ \delta \in (0, \delta_5)$ there exists  $s_5(b_0,\delta) \ge 1$ such that for all $s_0 \ge s_5$, the following property is  valid: Assume $(q,b)(s) \in V_{\delta, b_0}(s)$ for all  $s \in [s_0,\bar s]$ for some $\bar s > s_0$, then, we have   
\beqtn\label{bound-N}
\left| P_n(\mathcal{N}) \right|  \leq CI^{-2\delta }(s), \forall  \text{ and } n  \in \{0,1,..., [M]\},
\eeqtn
for all $s \in \left[s_0, \bar s \right]$ and $0 \le n \le [M]$.
%\textcolor{blue}{Nejla: I think that as in \cite{BKnon94}, the bound in the right hand side is $I^{-2\delta}$, it is also the same in all the other Lemma of this part.}\\
%The function $N(q,y,s)$ can be decomposed for all $s\geq 1 $ and $|q|\leq 1$ as
%\begin{equation}
%\dsp\sup_{|y|\leq b(s)^{-\frac{1}{2k}}}\left |N-\dsp ??\sum_{0\leq j\leq K} c_j\left( (e_b q)_+ \right )^j 
%\right |\leq C |q|^{ K+1}+C I_s^{-\delta(K)}\; K ??
%\end{equation}
\end{lemma}
\begin{proof} We argue as in \cite{BKnon94}. 
First, let us recall  the nonlinear term  $\mathcal{N}$ and $ P_n(\mathcal{N}) $ defined as in \eqref{nonlinear-term} and  \eqref{defi-q_m}, respectively. The  main goal is to use the estimates defined in $V_{\delta, b_0}(s)$ to get an improved bound on $P_n(\mathcal{N})$. Firstly, we recall the following identity 
\begin{equation}\label{e-b-identity-L}
    e_b(y)=(p-1)^{-1}\left (\sum_{\ell=0}^{L}\left (-\frac{by^{2k}}{p-1}\right )^\ell + \left (-\frac{by^{2k}}{p-1}\right )^{L+1}e_b(y)\right), \forall L \in \mathbb{N}^*.
\end{equation}
\iffalse
The main goal of this Lemma is to improve the bounds  of  the $P_n(\mathcal{N})$.  The main idea  follows the fact that, the projection only is  affected on a very small   neighborhood of $0$, for instant $[0,I^{-\delta}(s)]$ with $ I^{-\delta}(s) \to 0 \text{ as } s \to +\infty$, and the  rest part on  $[I^{-\delta}, +\infty)$  we  be ignored.  In addition to that, on the  main interval, we can apply Taylor expansion as  well to get some cancellation to  the nonlinear term, that finally  deduces a good estimate on the projection $ P_n(\mathcal{N})$. Let us start to the detail of the proof: We now recall of necessary estimates for the proof in the below. \fi
From the fact that $ (q,b)(s) \in  V_{\delta, b_0}(s) $ for all   $s \in [s_0, s_1]$, then get the following 
\begin{eqnarray}
\left| e_{b}(y) q(y) \right|  = | e_{b}(y)| \left| \left( \sum_{m=0}^M q_m(s) H_m(y,s) +q_-(y,s) \right) \right|  \le C I^{-\delta }(s) (1 + |y|^M), \label{rough-esti-e-b-q}
\end{eqnarray}
which implies 
\begin{eqnarray}
\left|\mathcal{N}(q)(y,s) \right| \le C  \left| 1 + e_{b}(y,s) q(y,s)  \right|^p  \le C[1+I^{-p-\delta}(s)(1 + |y|^{Mp})].   \label{rough-estimate-N-q}
\end{eqnarray}
By applying   Lemma \ref{small-integral-y-ge-I-delta} with $f(y) =\mathcal{N}(y)$ and $K =pM, \delta =0$, we  obtain 
\begin{eqnarray}\label{projection-xc-N}
\left| \int_{|y| \ge 1} \mathcal{N}(y,s) H_n(y,s)\rho_{s}(y)dy \right| \le C e^{-\frac{I(s)}{8}}, \forall n \in \{0,1,...,[M]\},
\end{eqnarray}
then it follows 
\begin{eqnarray}
\left| \int_{|y| \ge 1} \mathcal{N}(y,s) H_n(y,s)\rho_{s}(y)dy \right| \le  C I^{-2\delta -2n}, \forall n \in \{0,1,...,[M]\}.\label{esti-integral-N-lea-I-delta}
\end{eqnarray}
provided that $s_0 \ge s_{1,1}(\delta, M)$. 
We here claim that the following estimate 
\begin{eqnarray}
\left|\int_{|y| \le 1} \mathcal{N}(y,s) H_n(y,s)\rho_{s}(y)dy \right| \le C I^{-2\delta -2n}, \forall n \in \{ 0,1..., [M]\}, \label{integral-N-H-n-les-I-delta}
\end{eqnarray}
directly concludes the proof of Lemma \ref{projecion-H-n-N}. Indeed, let us assume that  \eqref{esti-integral-N-lea-I-delta} and  \eqref{integral-N-H-n-les-I-delta} hold, then we  derive
$$ \left| \left\langle \mathcal{N}, H_n(y,s)  \right\rangle_{L^2_{\rho_s} } \right| \le C I^{-2\delta -2n}(s), \forall n \in \{0,1...,[M]\},  $$
which  implies 
$$ |P_n(\mathcal{N})| \le CI^{-2\delta}(s), \forall s \in [s_0,s_1] \text{ and } n \in \{0,1...,[M]\}, $$
since  and it  concludes \eqref{projecion-H-n-N} and also Lemma \ref{projecion-H-n-N}. Now, it remains to prove  \eqref{integral-N-H-n-les-I-delta}.  From \eqref{rough-esti-e-b-q}, we have   
\begin{eqnarray}\label{esti-e-b-q-yle-1}
\left| e_{b(s)}(y) q(y,s) \right| \le CI^{-\delta}(s), \forall s \in [s_0,s_1] \text{ and } |y| \leq 1. 
\end{eqnarray}
then, we apply  Taylor expansion to  function   $\mathcal{N}(q)$   in  the variable $z = q e_{b}$ (here we usually  denote $b$ standing  for $b(s)$) and we get  
\begin{eqnarray}
\mathcal{N}(q)&=&|1+e_bq|^{p-1}(1+e_bq)-1-p e_b q =\sum_{j=2}^{K}c_j (e_bq)^j + R_K, \label{defi-R-K}
\end{eqnarray}
where $K $ will be fixed later and the reader should bear in mind that we only consider $|y| \leq 1$ in this part. For the remainder $R_K$, we derive from 
\begin{eqnarray}
\left| R_K(y,s) \right| \le C \left| e_{b}(y)q(y,s) \right|^{K+1} \le C  I^{-\delta (K+1)}(s).\label{property-R-K}
\end{eqnarray}
Besides that,  we recall from \eqref{decomposition-q2} that $ q = q_+ + q_-$ and we   have then express
\begin{eqnarray*}
\sum_{j=2}^K c_j (e_{b} q )^j = \sum_{j=2}^{K}d_{j,j} (e_bq_+)^j + \sum_{j=2}^K \sum_{\ell=0}^{i-1} d_{j,\ell} e_{b}^j(q_+)^\ell (q_-)^{j-\ell}
= A + S,
\end{eqnarray*}
where
\begin{eqnarray}
A = \sum_{j=2}^{K}d_{j} (e_bq_+)^j \text{ and } S =\sum_{j=2}^K \sum_{\ell=0}^{j-1} \tilde d_{j,\ell} e_{b}^j(q_+)^\ell (q_-)^{j-\ell},  \text{ for some } d_j, \tilde d_{j,\ell} \in \mathbb{R}.\label{defi-A-S}
\end{eqnarray}
From the above expressions, we can decompose    $\mathcal{N}$ by 
$$ \mathcal{N} = A + S + R_K,$$
and we also have  
\begin{eqnarray*}
\int_{|y| \le  1 } \mathcal{N}(y,s) H_n(y,s) \rho(y) dy &=&  \int_{|y| \le 1} A H_n(y,s) \rho(y) dy +\int_{|y| \le 1}  S H_n(y,s) \rho(y) dy \\
&+& \int_{|y| \le 1} R_K H_n(y,s) \rho(y) dy.
\end{eqnarray*}
- \textit{The integral for $R_K$}    Note that  $H_n$ defined  in \eqref{eigenfunction-Ls} satisfies 
$$ \left| H_n(y,s) \right| \le C (1 + |y|^n) \le C, \forall  |y| \le  1, $$
hence, it follows from \eqref{property-R-K} that 
\begin{eqnarray}
\left|\int_{|y| \le 1  } R_K(y,s) H_n(y,s)\rho_s(y) dy\right| &\le & C   I^{-\delta(K+1)}(s)  \int_{|y| \le 1}  e^{-\frac{I^2(s)y^2}{4}} I(s) dy \nonumber\\
& \le & C I^{1-\delta (K+1)}(s) \le CI^{-2\delta - 2n }(s), \forall s \in [s_0,s_1],\label{estimate-on-integral-R-K}
\end{eqnarray}
provided that $K \ge K_1(\delta,M)$.
  
   \noindent
- \textit{ The integral for   $S$:} Since $(q,b)(s) \in V_{\delta,b_0}(s)$, for all $s \in [s_0,\bar s]$, we can estimate as follows
\begin{eqnarray*}
\left| q_+(y,s)  \right|^\ell + |q_-(y,s)|^\ell  =  \left| \sum_{m=0}^M q_m(s) H_{m}(y,s) \right|^\ell  + C I^{-\ell \delta}(s)(I^{-M}(s)+|y|^M)^\ell   \le C  I^{-\ell \delta}(s),  
\end{eqnarray*}
for all $ |y|  \le 1 \text{ and } \ell \in \mathbb{N}.$ Regarding to \eqref{defi-A-S}, we can  estimate  as follows 
\begin{eqnarray*}
\left| S(y,s)    \right| \le C \left(\left|q_+(y,s) \right| |q_-(y,s)| +|q_-(y,s)|^2 \right)\le C  I^{-2\delta}(s) ( I^{-M}(s) + |y|^M ),
\end{eqnarray*}
 provided that $s_0 \ge s_{1,3}(K)$. Thus, we derive
\begin{eqnarray*}
\left| \int_{|y|  \le 1} S(y,s) H_n(y,s) \rho_s(y) dy \right| \le C I^{-2 \delta}(s) \int_{|y| \le 1} \left(I^{-M}(s) + |y|^M \right) |H_n(y,s)| e^{-\frac{I^2(s) y^2 }{4}} I(s) dy.
\end{eqnarray*}
Accordingly to \eqref{eigenfunction-Ls} and changing of variable $z = I(s) y$, we have
\begin{eqnarray}
& & \hspace{-0.8cm} \int_{|y| \le 1} \left(I^{-M}(s) + |y|^M \right) |H_n(y,s)| e^{-\frac{I^2(s) y^2 }{4}} I(s) dy \label{changin-variable-z}\\
&=& I^{-M-n}(s) \int_{|z| \le I(s)} (1+|z|^M) |h_n(z)| e^{-\frac{|z|^2}{4}} dz \le C I^{-M-n}(s). \nonumber
\end{eqnarray}
Finally, we have 
\begin{eqnarray}
\left| \int_{|y| \le 1} S(y,s) H_n(y,s) \rho_s(y) dy  \right| \le CI^{-2\delta-2n}(s), \forall n \le M,\label{integral-S-H-n} \forall s \in [s_0,s_1], 
\end{eqnarray}
provided that $s_0 \ge s_{1,3}(K)$.

\noindent 
- \textit{The integral for $A$}:  From \eqref{decomposition-q2}  and \eqref{eb0-modulation}, we write 
\iffalse 
\begin{eqnarray*}
e_{b}(y) = \sum_{\ell =0}^{K-1} E_\ell b^\ell y^{2\ell k } + O(y^{K 2k }),  \text{ with } E_j \in \R.
\end{eqnarray*}
Then, with  $q_+   $  defined as in \eqref{decomposition-q2}, we conclude \fi 
\begin{eqnarray*}
\left( e_{b}q_+   \right)^j =  \left( \sum_{\ell =0}^{K-1} E_\ell b^\ell y^{2\ell k } \right)^j \left(  \sum_{m=0}^{\left[ M \right]} q_m H_m \right)^j + O(|q_+|^2 y^{K(2k)} ), \forall  j \ge 2.
\end{eqnarray*}
By using the technique in \eqref{changin-variable-z} (changing variable $z = I(s) y$), we obtain 
\begin{eqnarray}
\int_{|y|  \le 1  } |y|^{K(2k)} |q_+|^2(y) \rho_s (y) dy  & \le  & C I^{-2\delta}(s) \int_{|y| \le 1  }   |y|^{K(2k)} \left( \sum_{m=0}^{\left[ M\right]} \left|H_m(y,s) \right| \right)^2 \rho_s dy \nonumber\\
&\le & I^{-2\delta -K(2k)  }(s)  \le C I^{-2\delta - 2n}(s),    \label{estimates-the-errors-y-K}
\end{eqnarray}
provided that $ K \ge K_2(\delta, M)$ large enough. In addition, we derive from  $H_m$'s definition defined  in \eqref{eigenfunction-Ls}  that  
\begin{eqnarray*}
\left( \sum_{\ell=0}^{K-1} E_\ell b^\ell y^{(2k)\ell} \right)^j \left(  \sum_{m=0}^{\left[ M \right]} q_m H_m \right)^j = \sum_{k=0}^{L} \mathcal{A}_k(s) y^k \text{ where } L = j\left( \left[M \right] + (K-1)(2k)\right), 
\end{eqnarray*}
and $\mathcal{A}_j$ satisfying
\begin{eqnarray*}
\left| \mathcal{A}_j(s) \right| \le C I^{-2\delta}(s).
\end{eqnarray*}
Now, we  apply Lemmas \ref{lemma-scalar-product-H-m} and \ref{small-integral-y-ge-I-delta} to deduce 
\begin{eqnarray}
\left| \int_{|y| \le 1 } \left( \sum_{n=0}^{K-1} E_n b^n y^{(2k)n} \right)^j \left(  \sum_{m=0}^{ [M] } q_m H_m \right)^j H_n(y,s)\rho_s(y) dy  \right|  \le C I^{-2\delta -2n} (s). \label{estinate-polynomial-q-+}
\end{eqnarray}
Thus, we get
\begin{eqnarray}
\left| \int_{|y| \le 1} A(y,s) H_n(y,s) \rho_s(y) dy  \right| \le CI^{-2\delta-2n}(s), \forall n \le M, \forall s \in [s_0,s_1]. \label{estimate-A}
\end{eqnarray}
According to \eqref{estimate-on-integral-R-K}, \eqref{integral-S-H-n} and \eqref{estimate-A}, we have
\begin{eqnarray}
\left| \int_{|y| \le 1} \mathcal{N}(q) H_n(y,s) \rho_s(y) dy \right| \le C I^{-2\delta -2n}(s),\label{estimate-}
\end{eqnarray}
provided that $s_0 \ge s_{1,3}(K)$,  and $K \ge K_2$. Thus, \eqref{integral-N-H-n-les-I-delta} follows which concludes the conclusion of the Lemma.
\end{proof}

\medskip

\textbf{d) Fourth term $b'(s)\mathcal{M} (q)$.} Let us consider  $\mathcal{M}$'s definition  that 
\[\mathcal{M}(q)=\frac{p}{p-1}y^{2k} (1+e_bq),\]
we have then  the following result:
\begin{lemma}\label{lemma-P-n-M} Let $b_0 >0$, then there exists $\delta_6(b_0)$ such that for all $  \delta \in (0, \delta_6)$, then there exists $ s_6(\delta, b_0) \ge 1$ such that for all $s_0 \ge s_6$ the following folds: Assume  $(q,b)(s) \in V_{\delta,b_0}(s), \forall  s \in [s_0, \bar s]$ for some $\bar s$ arbitrary, then it holds that
\begin{equation}\label{project-P-n-M}
    P_{n} \left( \mathcal{M} (q) (s) \right) = \left\{
    \begin{array}{rcl}
     \frac{p}{p-1}  + O(I^{-\delta}(s)) & \text{ if }   &   n = 2k \\
      O(I^{-\delta}(s))    & \text{ if } & n \ne 2k, n \in \{0,1,...,[M]\}
    \end{array}
    \right. .
\end{equation}
for all $s \in [s_0, \bar s]$.
\end{lemma}

\begin{proof}
We  firstly  decompose as  follows
$$ \left\langle \mathcal{M}, H_n(y,s) \right\rangle_{ L^2_{\rho_s}}  =  \left\langle  \frac{p}{p-1} y^{2k} , H_n(y,s) \right\rangle_{ L^2_{\rho_s}}  + \left\langle  \frac{p}{p-1}y^{2k}e_b(y) q  , H_n(y,s) \right\rangle_{ L^2_{\rho_s}}.$$
From \eqref{eigenfunction-Ls}, we get the following 
\begin{eqnarray}\label{part-1-M}
\left\langle  \frac{p}{p-1} y^{2k} , H_n(y,s) \right\rangle_{ L^2_{\rho_s}}
= \frac{p}{p-1}  \left\{ \begin{array}{rcl}
   \|H_{2k}\|^2_{L^2_{\rho_s}}   & \text{ if } &  n=2k\\[0.2cm]
   O(I^{-2k-2}(s))   & \text{ if } & n < 2k \\
0 & \text{ if   } &    n > 2k
\end{array}   \right. ,
\end{eqnarray}
Now we focus on the scalar product
$$ \left\langle  \frac{p}{p-1}y^{2k}e_b(y) q  , H_n(y,s) \right\rangle_{ L^2_{\rho_s}} .$$
We decompose 
\begin{eqnarray*}
\left\langle  \frac{p}{p-1}y^{2k}e_b(y) q  , H_n(y,s) \right\rangle_{ L^2_{\rho_s}} &=& \int_{|y| \le 1} \frac{p}{p-1} y^{2k} e_n(y) q H_n(y,s) \rho_s (y) dy \\
&+&\int_{|y| \ge 1} \frac{p}{p-1} y^{2k} e_n(y) q H_n(y,s) \rho_s (y) dy.  
\end{eqnarray*}
Since $q \in V_{ \delta, b_0}(s)$ for all $ s \in [s_0,s^*]$,   the following estimate holds
\begin{eqnarray*}
\left|  \frac{p}{p-1} y^{2k} e_b(y) q \right| \le C  I^{-\delta}(s) |y|^{2k} (1+ |y|^{M}).
\end{eqnarray*}
Using  Lemma \ref{small-integral-y-ge-I-delta}, we conclude   
\begin{eqnarray}
& & \left| \int_{|y| \ge 1 } \frac{p}{p-1} y^{2k} e_b(y) q H_n(y,s) \rho_s (y) dy \right|  \label{integral_M-I-ge-I-delta} \\
&\le & C I^{-\delta} e^{-\frac{1}{8} I(s)} \le C I^{-2\delta}(s), \forall s \in [s_0,s^*],\nonumber 
%\text{ and } n \in \{ 0,1...,M\}, \nonumber
\end{eqnarray}
provided that $s_0 \ge s_3(\delta)$.

%%%%%
%On the other hand, on the domain $[0,1]$, we can  apply Taylor expansion to $e_b$ defined as in \eqref{e-b-identity-L} as follows
%\begin{eqnarray*}
%e_b(y)=(p-1)^{-1}\left (\sum_{\ell=0}^{L}\left (-\frac{by^{2k}}{p-1}\right )^\ell + \left (-\frac{by^{2k}}{p-1}\right )^{L+1}e_b(y)\right), \forall L \in \mathbb{N}^*.
%\end{eqnarray*}
%Then, we can write the following 
%\begin{eqnarray*}
%\frac{p}{p-1} y^{2k} e_b(y)  q = \sum_{i=0}^M \sum_{j=0}^{K+1} m_j b^j  q_i(s) H_i(y,s)  + O(|q_-|)  
%\end{eqnarray*}
Let us decompose 
\begin{eqnarray*}
\frac{p}{p-1} y^{2k} e_b(y)  q = \frac{p}{p-1} y^{2k} e_b(y)  q_+ + \frac{p}{p-1} y^{2k} e_b(y)  q_- .
\end{eqnarray*}
Since $q  \in V_{\delta,  b_0}(s)$ and $e_b$ bounded, we get
\begin{eqnarray*}
\left|   \frac{p}{p-1} y^{2k} e_b(y)  q_-   \right| \le C I^{-\delta}(s) |y|^{2k}  (I^{-M}(s) + |y|^M).  
\end{eqnarray*}
By the same technique in \eqref{changin-variable-z}, we obtain
\begin{equation}\label{bound-for-y-2k-e-b-q-}
    \left| \int_{|y| \le 1  }     \frac{p}{p-1} y^{2k} e_b(y)  q_- H_n (y,s) \rho_s(y) dy  \right|  \le CI^{-2\delta -2n}(s), \forall s \in [s_0,s^*] \text{ and } n \in \{ 0,1...,[M]\}.
\end{equation}
In addition,  using \eqref{decomposition-q2}  and \eqref{eb0-modulation}, we write 
\begin{eqnarray*}
\frac{p}{p-1} y^{2k}  e_b(y)  q_+ = \sum_{i=0}^M \sum_{j=1}^{K} m_{i,j} b^j q_i(s) y^{2kj} H_i(y,s) + O\left( I^{-\delta}(s) y^{(K+1)2k} (1 + |y|^M )  \right).
\end{eqnarray*}
Repeating the technique in  \eqref{changin-variable-z} (changing variable $z = I(s) y$), we obtain 
\begin{eqnarray*}
\left| \int_{|y| \le 1} I^{-\delta}(s) y^{(K+1)2k} (1 + |y|^M )  H_n(y,s) \rho_s(y) dy   \right|  \le CI^{-2\delta -2n} (s), \forall s \in [s_0,s^*], n \in \{ 0,1,..., M\},
\end{eqnarray*}
provided that $ K $ large enough.  Besides that, we 
use the fact that $  q \in V_{\delta, b_0}(s) $ to get
\begin{eqnarray*}
\left| q_j(s)   \right| \le CI^{-\delta}(s),
\end{eqnarray*}
and $H_i$ can be written  by a polynomial in $y$, we apply Lemma   \ref{lemma-scalar-product-H-m} and Lemma \ref{small-integral-y-ge-I-delta},  we derive
\begin{eqnarray*}
& & \left| \int_{|y| \le 1} \left( \sum_{i=0}^M \sum_{j=1}^{K} m_{i,j} b^j q_i(s) y^{2kj} H_i(y,s) \right)  H_n(y,s) \rho_s(y) dy   \right| \\
&\le  & CI^{-\delta -2n}(s), \forall s \in [s_0,s^*]  \text{ and  }  n \in \{ 0,1,..., [M]\}.
\end{eqnarray*}
Finally, we get
\begin{equation}\label{bound-y-2k-e-b-q+}
      \left| \int_{|y| \le 1  }     \frac{p}{p-1} y^{2k} e_b(y)  q_+ H_n (y,s) \rho_s(y) dy  \right|  \le CI^{-\delta -2n}(s), \forall s \in [s_0,s^*] \text{ and } n \in \{ 0,1...,[M]\}.
\end{equation}
Now, we combine \eqref{bound-for-y-2k-e-b-q-} with \eqref{bound-y-2k-e-b-q+}  to imply
\begin{equation}\label{integra-M-y-le-I-delta}
    \left| \int_{|y| \le 1  }     \frac{p}{p-1} y^{2k} e_b(y)  q H_n (y,s) \rho_s(y) dy  \right|  \le CI^{-\delta -2n}(s), \forall s \in [s_0,s^*] \text{ and } n \in \{ 0,1...,[M]\}.
\end{equation}
We use  \eqref{integral_M-I-ge-I-delta} and \eqref{integra-M-y-le-I-delta} to conclude 
\begin{equation}\label{part-2-M}
  \left|  \left\langle  \frac{p}{p-1}y^{2k}e_b(y) q  , H_n(y,s) \right\rangle_{ L^2_{\rho_s}}  \right| \le CI^{-\delta-2n}(s), \forall s \in [s_0,s^*] \text{ and } n \in \{ 0,1...,[M]\}.
\end{equation}
Finally, by \eqref{part-1-M} and \eqref{part-2-M} we conclude the proof of the Lemma. \end{proof}

\medskip
\textbf{e) Fifth term $\mathcal{D}_s (q)$}\\
\begin{lemma}[Estimation of $P_n(\mathcal{D}_s)$] \label{lemma-P-n-mathcal-D} Let $b > 0$, then there exists $\delta_7(b_0) > 0$  such that for all $\delta \in (0,\delta_7)$ , there exists $s_7(\delta, b_0)$  such that for all $s_0 \ge s_7$, the following property holds: Assume  $(q,b)(s) \in V_{\delta, b_0}(s)$ for all $ s \in [s_0,\bar s]$ for some $\bar s \ge s_0$, then we have 
\begin{equation}\label{projec-P-n-mathcal-D}
    \left| P_n(\mathcal{D}_s(q)) \right|  \leq C  I^{-2\delta} (s), \text{ for all } s \in [s_0,\bar s],
\end{equation}
for all $ 0 \le  n \le M $.
\end{lemma}
\begin{proof}
 Let us now recall  from \eqref{equation-Ds}  that  
\[\mathcal{D}_s(\nabla q)=-\frac{4pkb}{p-1}I_s^{-2} y^{2k-1}e_b\nabla q.\]
From \eqref{defi-q_m} and \eqref{scalar-product-hm}, it is sufficient to estimate to 

\begin{eqnarray*}
  \left\langle  \mathcal{D}_s, H_n(y,s)     \right\rangle_{L^2_{\rho_s} }  
=  \int_{\mathbb R} \left(-\frac{4pkb}{p-1} I^{-2}(s) y^{2k-1} e_b \nabla q H_n(y,s) \rho_s(y) dy    \right)   
\end{eqnarray*}
From the fact that $\nabla (H_n)=nH_{n-1}, \rho_s(y) = \frac{I(s)}{4\pi} e^{-\frac{I^2(s) y^2}{4}}$, we use  integration by parts to derive 
\begin{eqnarray*}
 & & \langle  \mathcal{D}_s, H_n(y,s)     \rangle_{L^2_{\rho_s} } \\
 & = & \frac{4pkb}{p-1}I^{-2}(s)\left (\int \nabla (y^{2k-1}e_b)q H_n \rho_s(y)dy,\right.  + n\int y^{2k-1}e_b q H_{n-1} \rho_s dy
 \left . -\frac 1 2I^2(s)\int y^{2k} e_b q y H_n  \rho_s dy \right ).
\end{eqnarray*}
Then, we explicitly  write the scalar product by four integrals as follows 
\begin{eqnarray*}
 \left\langle  \mathcal{D}_s, H_n(y,s)     \right\rangle_{L^2_{\rho_s} }  &=&\frac{4pkb}{p-1}I^{-2}(s)\left \{(2k-1)\int y^{2k-2}e_bq H_n \rho_s(y)dy\right .
-2kb\int y^{4k-2}  e_b^{2}q H_n \rho_s(y)dy\\
&&+n\int y^{2k-1}e_b q H_{n-1} \rho_s dy
\left . -\frac 1 2I^2(s)\int y^{2k} e_b q y H_n  \rho_s dy \right\}.
\end{eqnarray*}
By the technique established in  Lemma \ref{lemma-P-n-M}, we  can prove 
$$ \left|  \left\langle  \mathcal{D}_s, H_n(y,s)     \right\rangle_{L^2_{\rho_s} }  \right|  \le C  I^{-2\delta -2n}, \forall s \in [s_0,s^*], \text{ and  }  n \in \{0,1,...,[M] \}. $$
which concludes  \eqref{projec-P-n-mathcal-D} and the conclusion of  the Lemma follows.  \end{proof} 
\iffalse
Then we write
$$
\begin{array}{rcl}
P_n(\mathcal{D}_s)&=& \dsp \int_{|y|\leq b(s)^{-\frac{1}{2k}}} \mathcal{D}_s H_n\rho_s dy+ \int_{|y|\geq b(s)^{-\frac{1}{2k}}} \mathcal{D}_s H_n\rho_s dy \\
     & =&\mathcal{D}_1+\mathcal{D}_2.
\end{array}
$$
Arguing as in the proof of Claim \ref{estimation-M1-M2} and using the properties of our shrinking set \eqref{definition-shrinking-set},  we obtain
\beqna
|P_n(\mathcal{D}_s)|  \leq C  I^{-2\delta}(s).
\eeqna
\fi 
\medskip

\textbf{f) Sixth term $\mathcal{R}_s (q)$}
\begin{lemma}[Estimation of $P_n(\mathcal{R}_s)$]\label{Lemma-Rs-n} Let $b_0 > 0$, then there exists  $\delta_8(b_0) >0$  such that for all $ \delta \in (0, \delta_8)$ there exists $s_8(b_0, \delta) \ge 1$ such that for all $ s_0 \ge s_8$, the following holds
 \begin{equation}\label{bound-P-n-mathcal-R-s}
 \left| P_n(\mathcal{R}_s(q)) \right|  \leq C   I^{-2\delta}(s),
 \end{equation}
 for all $s \in [s_0,\bar s]$ and $0 \le n \le M$. 
\end{lemma}
\begin{proof}
 The technique is quite the same as the others terms in  above.   Firstly, we  write   $\mathcal{R}_s$'s definition given in   \eqref{equation-Rs} as follows 
\[
\mathcal{R}_s(q)= I^{-2}(s)y^{2k-2}\left (\alpha_1+\alpha_2 y^{2k}e_b+(\alpha_3+\alpha_4 y^{2k}e_b)q \right),
\]
then, we have   the following 
\begin{eqnarray*}
P_n(\mathcal{R}_s) = \frac{\left\langle \mathcal{R}_s, H_n(y,s) \right\rangle_{L^2_{\rho_s}} }{ \left\|   H_n(s),  \right\|_{L^2_{\rho_s}}^2 },
\end{eqnarray*}
where $\left\|   H_n(s) \right\|_{L^2_{\rho_s}}^2$ computed in   \eqref{scalar-product-hm}. In particular, we observe that  \eqref{bound-P-n-mathcal-R-s} immediately  follows by
\begin{eqnarray}
\left| \left\langle \mathcal{R}_s, H_n(y,s) \right\rangle_{L^2_{\rho_s}} \right| \le CI^{-2\delta- 2n}, \forall s \in [s_0,s^*] \text{ and }  \forall n \in \{ 0,1,...,[M]\}.\label{scalar-product-mathcal-R-H-n}
\end{eqnarray}
 Besides that the technique  of  the proof of  \eqref{scalar-product-mathcal-R-H-n} is proceed as in  Lemma \ref{lemma-P-n-M}. For that reason, we kindly refer the reader to check the details and we finish the proof  of the Lemma
\end{proof}
\iffalse
projecting $\mathcal{R}_s$ on $H_m$ gives:
\[\begin{array}{lll}
    P_n(\mathcal{R}_s) &=&\dsp \int_{|y|\leq b(s)^{-\frac{1}{2k}} }\mathcal{R}_sH_n \rho_s dy+ \int_{|y|\geq b(s)^{-\frac{1}{2k}} } \mathcal{R}_sH_n \rho_s dy\\
    &&=\mathcal{R}_1+\mathcal{R}_2.
     \end{array}
     \]
To get the estimation for $\mathcal{R}_1$ and $\mathcal{R}_2$, we proceed as in the proof of (i) of Claim \ref{estimation-M1-M2}.\\
For $\mathcal{R}_2$, we note from the definition of $\mathcal{R}$ that when $b(s)y^{2k}\geq 1$
\[|\mathcal{R}_s|\leq CI_s^{-2\delta}|y|^{M-2},\]
which allows us to conclude as in the proof of (ii) of Claim \ref{estimation-M1-M2}.
\fi

%%%%%%%
%%%%%%%
%%%%%%%

\textbf{Part 2: Proof of (i) and (ii) of Proposition \ref{proposition-ode}}:

\medskip
\noindent
\textit{- Proof of (i) of Proposition \ref{proposition-ode}:}\\
Combining Lemma \ref{Lemma-Pn_partialq}-\ref{Lemma-Rs-n} the estimates defined in $V_{\delta, b_0}(s)$, we obtain (i) of Proposition \ref{proposition-ode}
\[\forall n \in \{0,..[M]\},\;\;\left |\pa_s q_n-\left( 1-\frac{n}{2k} \right)q_n\right |\leq CI^{-2\delta}(s), \forall s \in [s_0, \bar s],\]
provided that $\delta \le \delta_3$ and $s_0 \ge s_3(\delta, b_0)$. Thus, we conclude item (i).

\medskip
\noindent
\textit{- Proof of (ii) of Proposition \ref{proposition-ode}: Smallness of the modulation parameter.
}\\
Let us recall the equation satisfied by $q$:
 \beqtn\label{equation-q-bis}
 \pa_s q =\mathcal{L}_s q+b'(s)\mathcal{M}(q) +\mathcal{N} (q)+\mathcal{D}_s(\nabla q)+\mathcal{R}_s(q),
\eeqtn
this part aims to obtain an estimation of the modulation parameter $b(s)$. For this we will project the equation \eqref{equation-q-bis} on $H_{2k}$ and take on consideration that $q_{2k}=0$, we obtain
\beqtn\label{modulation-equation}
0=\frac{p}{p-1}b'(s)\left (1+ P_{2k}(y^{2k}e_bq)\right )+P_{2k}(\mathcal{N}) +P_{2k}(\mathcal{D}_s)+P_{2k}(\mathcal{R}_s), 
\eeqtn
Using estimations given by equation \eqref{bound-N} and Lemmas \ref{lemma-P-n-M}, \ref{lemma-P-n-mathcal-D} and \ref{Lemma-Rs-n}, we obtain
\beqtn\label{inequality-b}
|b'(s)|\leq CI(s)^{-2\delta}=C e^{\delta\frac{1-k}{k}s},
\eeqtn
where $0<\delta\leq \min (\delta_j,5\le j\le 8 )$ is a strictly positive real, which gives us the smallness of the modulation parameter in i) of Proposition \ref{proposition-ode} and we obtain
\beqtn
b(s)\to b^*\mbox{ as }s\to \infty,\;\;  (t\to T).
\label{convegence-b-s}
\eeqtn
Integrating inequality \eqref{inequality-b} between $s_0$ and infinity, we obtain
\[ |b^*-b_0|\leq C e^{\delta\frac{1-k}{k}s_0},\]
we conclude that there exist $s_{9}$ such that dor  for $s_0\geq s_9$ big enough, we have 
\[\frac{3}{4} b_0\leq b^*\leq \frac{5}{4} b_0,\]
which is (ii) of Proposition \ref{proposition-ode}.

\subsection{The proof to item (iii)  of Proposition   \ref{proposition-ode}   }\label{proof-item-iii}

Here, we prove the last identity of Proposition \ref{proposition-ode}. As in the previous subsection, we proceed in two parts:
\begin{itemize}
\item In Part 1, we project equation \eqref{equation-q} using projector $P_-$ defined in \eqref{projector-P-} .
\item In Part 2, we prove the estimate on $q_-$ given by (iii) of Proposition \ref{proposition-ode}.
\end{itemize}

\textbf{Part 1: The projection of equation \eqref{equation-q} using the projector $P_-$.}
Let  $(q,b)$  be  solution to problem \eqref{equation-q} $\&$ \eqref{Modulation-condition} trapped in $V_{\delta, b_0}(s)$ for all $s \in [s_0, \bar s]$ for some $\bar s > s_0$.  Then, we have the following results:

\medskip
\textbf{First term $\pa_s q$.}\\
\begin{lemma} 
For all $s \in [s_0, \bar s]$, it holds that 
 \begin{equation}\label{esti-par-q-m-P-partial-s-q_-}
        P_-(\partial_s q)=\partial_s q_- - I^{-2}(1-\frac{1}{k})\sum_{n=[M]-1}^{[M]}(n+1)(n+2)q_{n+2}(s)H_n.
        %\left |P_-(\partial_s q)-\partial_s q_-\right |
    % \le C I^{-\delta } (s), \quad \forall s \in (s_0,s^*),
 \end{equation}
\end{lemma}

\begin{proof} We firstly have 
\[
\begin{array}{lll}
   P_-(\partial_s q)-\partial_s q_-
   &= &-\displaystyle \left (\partial_s q-P_-(\partial_s q)\right )+\left (\partial_s q-\partial_s q_-\right )  ,  \\
     &=&-\displaystyle \sum_{n=0}^{[M]} P_n(\partial_s q)H_n+\sum_{n=0}^{[M]}\partial_s (q_n H_n),\\
     &=&-\displaystyle \sum_{n=0}^{[M]} P_n(\partial_s q)H_n+\sum_{n=0}^{[M]}\partial_s q_n H_n+\sum_{n=2}^{[M]}q_n\partial_s H_n,
     % &=&-\displaystyle \sum_{n=0}^{[M]} P_n(\partial_s q)-\partial_s q_n.
\end{array}
\]
we recall by \eqref{formula-partia-s-Hn} that for all $n\ge 2$
  \[ \partial_s H_n (y,s) = n (n-1) \left(1 -\frac{1}{k} \right) I^{-2}(s)  H_{n-2}(y,s),\]
then by Lemma \ref{Lemma-Pn_partialq}, we obtain the desired result
\[P_-(\partial_s q)=\partial_s q_- -I^{-2}\left(1-\frac{1}{k}\right)\sum_{n=[M]-1}^{[M]}(n+1)(n+2)q_{n+2}(s)H_n.\]
\end{proof}

\textbf{Second term $\mathcal L_s q$.}\\
By the spectral properties given in Section \ref{Section-Spectral-properties-Ls}, we can write
\begin{lemma} For all $s \in [s_0, \bar s]$, it holds that 
\[P_-(\mathcal{L}_s q)=\mathcal{L}_s q_- -I^{-2} (1-\frac{1}{k}) \displaystyle\sum_{n=[M]-1}^{[M]}(n+1)(n+2)q_{n+2} H_n.\]

%I^{-2} (1-\frac{1}{k})\left ( M(M+1)q_{M+1} H_{M-1}+(M+1)(M+2)q_{M+2} H_{M} \right ).\]
\end{lemma}

\begin{proof}
We write

\[
\begin{array}{lll}
   P_-(\mathcal{L}_s q)-\mathcal{L}_s q_-
   &= &-\displaystyle \left (\mathcal{L}_s q-P_-(\mathcal{L}_s q)-\right )+\left (\mathcal{L}_s q-\mathcal{L}_s q_-\right )  ,  \\
     &=&-\displaystyle \sum_{n=0}^{[M]} P_n(\mathcal{L}_s q) H_n+ \mathcal{L}_s \left (q-q_-\right ),\\
     
      &=&-\displaystyle \sum_{n=0}^{[M]} P_n(\mathcal{L}_s q) H_n+ \sum_{n=0}^{[M]}q_n\mathcal{L}_s(H_n) .
\end{array}
\]
From \eqref{Ls-Hm}, we obtain
\[\begin{array}{lll}
\displaystyle \sum_{n=0}^{[M]}q_n\mathcal{L}_s (H_n)&=&\displaystyle q_0+(1-\frac{n}{2k})q_1 H_1+\sum_{n=2}^{[M]}q_n\left [(1-\frac{n}{2k})H_n+I^{-2} n(n-1)(1-\frac{1}{k})H_{n-2}\right ],\\
&=& \displaystyle\sum_{n=0}^{M} (1-\frac{n}{2k})q_n H_n+ I^{-2} (1-\frac{1}{k}) \displaystyle\sum_{n=0}^{M-2}(n+1)(n+2)q_{n+2} H_n
%&&+\displaystyle
\end{array},
\]
We deduce from Lemma \ref{Lemma-P-n-mathcal-L-s} that

\[P_-(\mathcal{L}_s q)-\mathcal{L}_s q_-=-  I^{-2} (1-\frac{1}{k})\left [ M(M+1)q_{M+1} H_{M-1} -(M+1)(M+2)q_{M+2} H_{M} \right ].\]
\end{proof}

\medskip

\textbf{Third term  $\mathcal{N}$.}\\
\begin{lemma}\label{lemma-estimation-P--N}
Let $b_0 >0$, then there exists $\delta_{10}(b_0)$ such that for all $  \delta \in (0, \delta_{10})$, then there exists $ s_{10}(\delta, b_0) \ge 1$ such that for all $s_0 \ge s_{10}$ the following folds: Assume  $(q,b)(s) \in V_{\delta,b_0}(s), \forall  s \in [s_0, \bar s]$ for some $\bar s$ arbitrary, then it holds that
\[|P_-(\mathcal{N})|\leq C\left (I(s)^{-2\delta}+I(s)^{-p\delta} \right )\left (I(s)^{-M}+|y|^M\right ). \]
\end{lemma}

\begin{proof} We argue as in \cite{BKnon94}. We recall from \eqref{nonlinear-term} that
\[\mathcal{N}(q)=|1+e_bq|^{p-1}(1+e_bq)-1-p e_b q.
\]
We proceed in a similar fashion as in the projection $P_n(\mathcal{N})$, we will give estimations in the outer region $|y|\geq 1$ and the inner region $|y|\leq 1$. Let us first define $\chi_0$ a $C_0^{\infty}(\R^+,[0,1])$, with $supp(\chi) \subset [0,2]$ and $\chi_0=1$ on $[0,1]$, we define
\beqtn\label{def-chi}
\chi (y)=\chi_0\left (|y|\right ). %(|y|b(s)^{-\frac{1}{2k}}\right ).
\eeqtn
Using the fact that 
\[\mathcal{N}= \chi \mathcal{N}+\chi^c\mathcal{N}, \]
we claim the following:
\begin{cl}\label{estimation-P-N}
\begin{eqnarray}
(i)\;\; \left |P_-(\chi^c \mathcal{N} )\right |&\leq &C I(s)^{-\delta p}\left (I(s)^{-M}+|y|^M\right ),\\
(ii)\;\; \left |P_-(\chi \mathcal{N} )\right |&\leq &C I(s)^{-2\delta}\left (I(s)^{-M}+|y|^M\right ).
\end{eqnarray}
\end{cl}

\begin{proof} First, we will estimate $P_-(\chi^c \mathcal{N})$, then $P_-(\chi \mathcal{N})$ and conclude the proof of the lemma.\\
(i) Let us first write

\[
\begin{array}{lll}
P_-(\chi^c \mathcal{N})&=&\chi^c \mathcal{N} -\sum_{n\leq [M]+1} P_n(\chi^c \mathcal{N})H_n\\
&=&\dsp \chi^c \mathcal{N} -\sum_{n\leq [M]+1}\frac{\int_{|y|\geq 1 } \mathcal{N} H_n \rho_s dy}{\|H_n\|_{L^{2}_{\rho_s}}} H_n, 
\end{array}
\]
using the definition of the shrinking set we can write
\[|\chi^c(\mathcal{N})|\leq |\chi^c (CI^{-\delta}e_b|y|^M)^p|=|\chi^c \left (CI^{-\delta}(e_by^{2k})|y|^{\frac{2k}{p-1}}\right )^p|,\]
by the fact that $|e_by^{2k}|\leq C$ and $M=\frac{2kp}{p-1}$, we have
\[|\chi^c(\mathcal{N})|\leq CI^{-\delta p}|y|^M \]
Then using \eqref{projection-xc-N} we deduce (i) of Claim \ref{estimation-P-N}: 
\beqtn
|P_-(\chi^c \mathcal{N})|\leq CI(s)^{-\delta p}\left ( I(s)^{-M}+|y|^M\right ).
\label{bound-P-Ncchi}
\eeqtn
(ii) In the inner region $|y|\leq 1$, we proceed as the in the proof of Lemma \ref{projecion-H-n-N}. For $|y|\leq 1$, using the Taylor expansion as in \eqref{defi-R-K}, we write

\[\chi \mathcal{N}=\chi\left (A+S+R_K\right ),\]
where $A$ and $S$ are given by \eqref{defi-A-S}
\[
A =\chi \sum_{j=2}^{K}d_{j} (e_bq_+)^j \text{ and } S =\chi \sum_{j=2}^K \sum_{\ell=0}^{j-1} \tilde d_{j,\ell} e_{b}^j(q_+)^\ell (q_-)^{j-\ell},  \text{ for some } d_j, \tilde d_{j,\ell} \in \mathbb{R}.
\]
We get for $K$ large,
\beqtn\label{bound-RK2}
|\chi R_K|\leq \mathcal{I}_s^{-\delta} I(s)^{-M}.
\eeqtn
We proceed in a similar fashion as in he proof of Lemma \ref{projecion-H-n-N}, we write $A$ as

\beqtn
\begin{array}{lll}\label{A1-A2}
    A&=&\chi \sum_{j=2}^{K}d_j\left ( e_b q_+ \right )^j  \\
     & =&\dsp\chi \sum_{\textbf{n},p} c_{\textbf{n},p}b(s)^{\frac{p}{2k}}y^p \displaystyle \Pi_{i=1}^{[M]}q_{i}^{n_i}H_i^{n_i}+I(s)^{-2\delta}b(s)^{\frac{2k(L+1)}{2k}} y^{2k(L+1)}\chi Q,\\
     & = &A_1+A_2,
\end{array}
\eeqtn
where $\chi Q$ is bounded. Then, we divide the sum $A_1$ as follows

\beqtn
\begin{array}{lll}
    A_1&=&\chi\dsp \sum_{\textbf{n},p} c_{\textbf{n},p}b(s)^{\frac{p}{2k}}y^p \displaystyle \Pi_{i=1}^{[M]}q_{i}^{n_i}H_i^{n_i},\\
   & =&\dsp \chi\dsp \sum_{\textbf{n},p, p+\sum n_i\leq M} c_{\textbf{n},p}b(s)^{\frac{p}{2k}}y^p \displaystyle \Pi_{i=1}^{[M]}q_{i}^{n_i}H_i^{n_i}
    +\dsp \chi\sum_{\textbf{n},p,p+\sum n_i> M} c_{\textbf{n},p}b(s)^{\frac{p}{2k}}y^p \displaystyle \Pi_{i=1}^{[M]}q_{i}^{n_i}H_i^{n_i}\\
&=& A_{1,1}+A_{1,2},
\end{array}
\eeqtn
 In the first sum, $A_{1,1}$,we replace $\chi=1-\chi^c$ by $-\chi^c$, since $1$ will not contribute to $A_-$. Using the fact that $|y|\geq 1$ and by \eqref{bound-b}, we get
\[\dsp \chi^c \left |y^p \Pi_{i=1}^{[M]}H_i^{n_i}\right |\leq C |y|^M.\]
Since $H_m$ is bounded as follows
 \[|H_m(y,s)|\leq C(I(s)^{-m}+|y|^m),\]
we obtain by \eqref{bound-b}
\[\dsp \chi \left |y^p \Pi_{i=1}^{[M]}H_i^{n_i}\right |\leq C(I(s)^{-M}+|y|^M). \]
We conclude by the definition of the shrinking set given by  \eqref{definition-shrinking-set}, that
\beqtn
|A_{1,2}|\leq CI(s)^{-2\delta}\chi(y) \left (I(s)^{-M}+|y|^{M} \right ).
\eeqtn
%In the second sum $A_{1,2}$, we replace $\chi=1-\chi^c$ by $-\chi^c$, since $1$ will not contribute to $A_-$. Using the fact that $|y|\geq b(s)^{-\frac{1}{2k}}$ and by \eqref{bound-b}, we get
%\[\dsp \chi^c \left |y^p \Pi_{i=1}^{M}H_i^{n_i}\right |\leq C |y|^M.\]
By the properties of the shrinking set and the bound for $q_-$, we obtain the bound for the term $A_2$, defined by  \eqref{A1-A2}, more precisely we have
\[ |A_2|\leq CI(s)^{-2\delta}\chi(y) \left (I(s)^{-M}+|y|^{M} \right ).\]
Then, we conclude that
\beqtn\label{bound-A-}
|P_-(A)|=|A_-|\leq C I(s)^{-2\delta}(I(s)^{-M}+|y|^M).
\eeqtn
which yields the conclusion of item (ii).
\end{proof}
Now, we return to the proof of the Lemma. We deduce by \eqref{bound-P-Ncchi}, \eqref{bound-RK2} and \eqref{bound-A-} the following estimation for $P_-(\mathcal{N})$
\beqtn
|P_-(\mathcal{N})|=|\mathcal{N}_-|\leq C (I(s)^{-2\delta}+I(s)^{-p\delta})( I(s)^{-M}+|y|^M),
\eeqtn
thus end the proof of Lemma \ref{lemma-estimation-P--N}.
\end{proof}

\textbf{Fourth term $b'(s)\mathcal{M} (q)$.}\\
\begin{lemma}\label{lemma-estimation-P--M}Let $b_0 >0$, then there exists $\delta_{11}(b_0)$ such that for all $  \delta \in (0, \delta_{11})$, then there exists $ s_{10}(\delta, b_0) \ge 1$ such that for all $s_0 \ge s_{11}$ the following folds: Assume  $(q,b)(s) \in V_{\delta,b_0}(s), \forall  s \in [s_0, \bar s]$ for some $\bar s$ arbitrary, then it holds that
\[|P_-(\mathcal{M})|\leq C I(s)^{-\delta}\left (I(s)^{-M}+|y|^M\right ). \]
\end{lemma}

We recall that 
\[\mathcal{M}=\frac{p}{p-1}y^{2k} (1+e_bq),\]
then, we can write

\[ P_-\left (\mathcal{M}(q)\right )=\frac{p}{p-1}P_-(y^{2k}e_bq).\]
Let us write 
\[P_-(y^{2k}e_bq)= P_-(\chi y^{2k}e_bq)+P_-(\chi^c y^{2k}e_bq), \]
we claim the following:
\begin{cl}\label{estimation-P-M}
\begin{eqnarray}
(i)\;\; \left |P_-(\chi^c y^{2k} e_b q )\right |&\leq &C I(s)^{-\delta}\left (I(s)^{-M}+|y|^M\right ),\\
(ii)\;\; \left |P_-(\chi y^{2k} e_b q)\right |&\leq & C I(s)^{-\delta}\left (I(s)^{-M}+|y|^M\right ).
\end{eqnarray}
\end{cl}
\begin{proof} Let us first write

\[
\begin{array}{lll}
P_-(\chi^c y^{2k} e_b q)&=&\chi^c y^{2k} e_b q -\sum_{n\leq [M]+1} P_n(\chi^c y^{2k} e_b q)H_n\\
&=&\dsp \chi^c y^{2k} e_b q -\sum_{n\leq [M]+1}\frac{\int_{|y|\geq b(s)^{-\frac{1}{2k}} } y^{2k} e_b q H_n \rho_s dy}{\|H_n\|^2_{L^{2}_{\rho_s}}}H_n, 
\end{array}
\]

When $ |y|\geq 1$, using \eqref{definition-shrinking-set}, we can write
\[|y^{2k}e_b q|\leq C |q|\leq \frac{CI(s)^{-\delta}}{b(s)}|y|^M \leq CI(s)^{-\delta}|y|^M .\]

ii) As for i), we Write
\[P_-(\chi y^{2k}e_b q)=\chi y^{2k} e_b q-\sum_{n\leq M+1} P_n(\chi^c y^{2k} e_b q).\]
By Lemma \ref{lemma-P-n-M} we have $\left |\sum_{n\leq M+1} P_n(\chi^c y^{2k} e_b q)  \|H_n\|^{-2}_{L^{2}_{\rho_s}} \right |\leq C I(s)^{-\delta}$. \\
We conclude using the definition of the shrinking set and we obtain the following estimation
\[|\chi y^{2k} e_b q|\leq C I(s)^{-\delta}. \]
\end{proof}

%\beqtn
%\begin{array}{lll}
%P_-(\mathcal{N})&=&
%\dsp \int_{|y|\leq b(s)^{-\frac{1}{2k}} } \sum_{n\geq M+1}\mathcal{N} H_n \rho_s dy+ \int_{|y|\geq b(s)^{-\frac{1}{2k}} }\sum_{n\geq M+1} \mathcal{N} H_n \rho_s dy\\
%    &&=\mathcal{N}_{1,-}+\mathcal{N}_{2,-}.
%\end{array}
%\eeqtn

%%%
\textbf{Fifth term $\mathcal{D}_s (\nabla q)$}
\begin{lemma}
 Let $b_0 >0$, then there exists $\delta_{12}(b_0)$ such that for all $  \delta \in (0, \delta_{12})$, then there exists $ s_{12}(\delta, b_0) \ge 1$ such that for all $s_0 \ge s_{12}$ the following folds: Assume  $(q,b)(s) \in V_{\delta,b_0}(s), \forall  s \in [s_0, \bar s]$ for some $\bar s$ arbitrary, then it holds that
 \[P_-(\mathcal{D}_s)\leq C I^{-2\delta}\left  (I(s)^{-M}+|y|^M\right ).\]
 \end{lemma}
\begin{proof}

Let us first write
\[\begin{array}{lll}
    P_-(\mathcal{D}_s) &=&\mathcal{D}_s-\dsp \sum_{n=0}^{[M]} P_n(\mathcal{D}_s)H_n,
    \end{array}\]

%\medskip

%\textbf{Old vesion as Bricomont an Kupiainen}
Since we are using the properties given by the shrinking set in Definition \ref{definition-shrinking-set}, it will be more convenient to estimate 
\beqtn 
\begin{array}{lll}
d &=&\dsp  \int_{\sigma}^{s}d\tau\mathcal{K}_{s,\tau }(y,z)\mathcal{D}_s(\nabla q). \\

\end{array}
\eeqtn
Using integration by parts, we obtain
\beqtn
\begin{array}{lll}
d&=&\dsp 4pkb(p-1)^{-1}\int_{\sigma}^{s}d\tau I({\tau})^{-2} \int dz
\partial _z\left (\mathcal{K}_{s,\tau }(y,z) e_b(z) z^{2k-1}\right )q(z,\tau),\\
&=&\dsp 4pkb(p-1)^{-1}\int_{\sigma}^{s}d\tau I({\tau})^{-2} \int dz
\mathcal{K}_{s,\tau }(y,z)\partial _z\left ( e_b(z) z^{2k-1}\right )q(z,\tau)\\

&&\dsp +pkb(p-1)^{-1}\int_{\sigma}^{s}d\tau I({\tau})^{-2} \int dz
\partial _z(\mathcal{K}_{s,\tau }(y,z)) e_b(z) z^{2k-1}q(z,\tau),\\
&=&d_1+d_2.
\end{array}
\label{decomposition-integral-d}
\eeqtn
 
%\begin{lemma}[Estimation of $P_-(d)$] There exists $\delta >0$, 
%\beqna
%|P_-(d)|  \leq I(s)^{-\delta}\left  (I(s)^{-M}+|y|^M\right ).
%\eeqna
%\end{lemma}
%\textit{Proof}: From equation \eqref{decomposition-integral-d} we write 
%\[P_-(d)=P_-(d_1)+P_-(d_2)\]
For the estimation of the firs term $d_1$, we argue in a similar fashion as in the projection of  $P_n(\mathcal{M})$, see Lemma \ref{lemma-P-n-M}. 
For the second term, we argue as in Bricomont Kupiainen \cite{BKnon94}. Indeed, we need to bound $\partial _z K_{s,\tau}$. From equations \eqref{Kernel-Formula} we obtain 
\beqtn
|\partial _z(\mathcal{K}_{s,\tau }(y,z)|\leq C L\mathcal{F}_{\frac 12 L^2}\left (e^{\frac{s-\tau}{2k}}y-z\right )\leq \frac{CI(s)}{\sqrt{s-\tau}}\mathcal{F}_{\frac 12 L^2}\left (e^{\frac{s-\tau}{2k}}y-z\right ),
\eeqtn
where
$L=\frac{I(s)^{2}}{(1-e^{-(s-\sigma)})}$, $\mathcal{F}$  defined by \eqref{Kernel-Formula-F} $\text{ and } I(s)=\dsp e^{\frac s2(1-\frac 1k)}$. Then, by Definition \ref{definition-shrinking-set}, we obtain
\[|d_2|\leq I(s)^{-1} I(s)^{-\delta}\]
%C I(s)^{-1} I(s)^{-\delta}\left  (I(s)^{-M}+|y|^M\right )\leq I(s)^{-\delta}\left  (I(s)^{-M}+|y|^M\right ).$\blacksquare$\]
and we conclude that there exist $\delta_?$ such that for all $0<\delta\leq \delta_?$,
\beqtn\label{P--mathcalDs-part1}
|d|\leq CI^{-2\delta }\left  (I(s)^{-M}+|y|^M\right ). 
\eeqtn

\medskip

On the other hand by Lemma \ref{projec-P-n-mathcal-D}, we obtain 
\beqtn\label{P--mathcalDs-part2}
|\sum_{n=0}^{[M]} P_n(\mathcal{D}_s)H_n|\leq C I^{-2\delta}\left  (I(s)^{-M}+|y|^M\right ).
\eeqtn
We conclude from \eqref{P--mathcalDs-part1}, \eqref{P--mathcalDs-part1}
that 
\[P_-(\mathcal{D}_s)\leq C I^{-2\delta}\left  (I(s)^{-M}+|y|^M\right ).\]
\end{proof}

\medskip
\textbf{Sixth term $\mathcal{R}_s(q)$}
\begin{lemma}\label{P--mathcal-Rs}
 Let $b_0 >0$, then there exists $\delta_{13}(b_0)$ such that for all $  \delta \in (0, \delta_{13})$, then there exists $ s_{13}(\delta, b_0) \ge 1$ such that for all $s_0 \ge s_{13}$ the following folds: Assume  $(q,b)(s) \in V_{\delta,b_0}(s), \forall  s \in [s_0, \bar s]$ for some $\bar s$ arbitrary, then it holds that
\beqna
|P_-(\mathcal{R}_s (q))|  \leq C I(s)^{-2\delta}\left  (I(s)^{-M}+|y|^M\right ).
\eeqna
\end{lemma}

\begin{proof} By \eqref{equation-Rs}
\[
\mathcal{R}_s (q)= I(s)^{-2}y^{2k-2}\left (\alpha_1+\alpha_2 y^{2k}e_b+(\alpha_3+\alpha_4 y^{2k}e_b)q \right),
\]
we proceed as for the estimation of  $P_-(\mathcal{M})$.
\end{proof}
%\[P_-(\mathcal{R}_s (q))=P_-(\chi\mathcal{R}_s (q) )+ P_-(\chi^c\mathcal{R}_s (q) ).\]
%By ?? and the following decomposition
%\begin{eqnarray}
%P_-(\chi\mathcal{R}_s(q) )&=&\chi\mathcal{R}_s (q)- \sum_{n\leq M+1} P_n(\chi\mathcal{R}_s(q)),\\
%P_-(\chi^c\mathcal{R}_s (q) )&=&\chi^c\mathcal{R}_s (q)-\sum_{n\leq M+1} P_n(\chi^c\mathcal{R}_s (q)),
%\end{eqnarray}

\textbf{Part 2: Proof of the identity (iii) in Proposition \ref{proposition-ode} (estimate on $q_-$)}
If we apply the projector $P_-$ to the equation of \eqref{equation-q}, we obtain 
\begin{eqnarray*}
\partial_s q_-=\mathcal{L}_s q_-+ P_-\left (\mathcal{N} (q) +\mathcal{D}_s(q)+\mathcal{R}_s (q)+ b'(s)\mathcal{M}(q)\right )
\end{eqnarray*}
Using the kernel of the semigroup generated by $\mathcal{L}_s$, we get for all $s\in [\tau, s_1]$
The integral equation of the equation above is 

\[
\begin{array}{lll}
q_-(s)&=&\mathcal{K}_{s\tau} q_-(\tau)\\
&&+\displaystyle \int_{\tau}^{s}  \mathcal{K}_{s s' }
\left (P_-\left [\mathcal{N} (q)+\mathcal{D}_s(\nabla q)+\mathcal{R}_s (q)+b'(s')\mathcal{M}(q)\right ]\right )ds'.
\end{array}
\]
Using Lemma \ref{lemma-estimation-K-phi}, we get
\[
\begin{array}{lll}
|q_-(s)|_{s}&\leq& e^{-\frac{1}{p-1}(s-\tau)}|q_-(\tau)|_{\tau}\\
&& +\dsp  \int_{\tau}^{s} e^{-\frac{1}{p-1}(s-s')} \left |P_-\left [\mathcal{N} (q)+\mathcal{D}_s(\nabla q)+\mathcal{R}_s (q)+b'(s')\mathcal{M}(q)\right ]\right |_{s} ds'
\end{array}
\]
By Lemma \ref{lemma-estimation-P--N}, Lemma \ref{lemma-estimation-P--M}, Lemma \ref{P--mathcal-Rs} , equations \eqref{P--mathcalDs-part1}, \eqref{P--mathcalDs-part2} and the smalness of the modulation parmeter $b(s)$ given by (ii) of Proposition \ref{proposition-ode}, we obtain
\[
\begin{array}{lll}
|q_-(s)|_{s} &\leq& e^{-\frac{1}{p-1}(s-\tau)}|q_-(\tau)|_{\tau}+\dsp  \int_{\tau}^{s} e^{-\frac{1}{p-1}(s-s')}  I(s')^{-\delta\frac{\min(p,2)+1}{2}} ds'.
\end{array}
\]
Then, for $\delta \le \delta_3$, it holds that  
$$ \left| q_-(s)\right|_s \le e^{-\frac{s-\tau}{p-1}} \left| q_-(\tau)\right|_\tau +  C \left( I^{-\frac{3}{2} \delta}(s) + e^{-\frac{s-\tau}{p-1}}  I^{-\frac{3}{2}\delta}(\tau)\right).   $$
which concludes the proof of the last identity of Proposition \ref{proposition-ode}.

%\bigskip

%\textbf{Old Version}
%If we apply the projector $P_-$ to the integral equation given by \eqref{integral-equation-q}, then we obtain

%\beqtn\label{integral-equation-q}
%q_-(y,\tau)=P_-(\mathcal{K}_{\tau\sigma} q(\sigma))+\dsp P_-\left ( \int_{\sigma}^{\tau}
%\mathcal{K}_{\tau s }
%\left [\mathcal{N} (q)+ b'(s)\mathcal{M}(q) +\mathcal{D}(q)+\mathcal{R} (q)\right ]ds \right ),
%\eeqtn
%then we can write
%\beqtn
%q_-(y,\tau)=\mathcal{K}_{\tau \sigma} q_-(\sigma)+\dsp P_-\left ( \int_{\sigma}^{\tau}
%\mathcal{K}_{\tau s }
%\left [\mathcal{N} (q)+b'(s)\mathcal{M}(q)+\mathcal{D}(q)+\mathcal{R} (q)\right ]ds \right ).
%\label{integral-equation-P-}
%\eeqtn
%Let us recall Lemma 4.1 from Bricomont and Kupiainen \cite{BKnon94}.
%\begin{lemma}
%\[\left \|K_{\tau\sigma} q_-\right \|_{s-\tau}\leq Ce^{-\frac{1}{p-1}(\tau-\sigma)}\left \| q_-\right \|_{\sigma}
%\]
%\end{lemma}
%\textit{Proof :} See proof of Lemma 4.1 from \cite{BKnon94}.

%We can easily see that $P_-\left (\dsp\int_{\sigma}^{\tau}
%\mathcal{K}_{\tau s } \mathcal{N} (q)\right )= \dsp\int_{\sigma}^{\tau}
%\mathcal{K}_{\tau s} P_-(\mathcal{N} (q) $, then equation \eqref{integral-equation-P-} can be written
%\[q_-( s)=\mathcal{K}_{\tau\sigma} q_-(\sigma)+\dsp  \int_{\sigma}^{\tau}
%\mathcal{K}_{\tau s }
%\left [P_-(\mathcal{N} (q))+P_-(b'(s)\mathcal{M}(q))+P_-(\mathcal{D}(q))+P_-(\mathcal{R} (q))\right ]ds .\]
%Using estimation given by Lemmas 

\appendix 
\section{Computation on Hermite polynomials}

\begin{lemma}\label{lemma-scalar-product-H-m}
 Let us consider $H_m$ defined as in \eqref{eigenfunction-Ls}, for some $m \ge 0$, let us consider 
 $$ f(y) = \sum_{j=0}^\ell f_j y^j, f_j \in \R.       $$
 Then, the following holds
 \begin{equation}\label{scalar-product--f-H-m}
     \left|\langle  f, H_m   \rangle_{L^2_{\rho_s}} \right|\leq 
     \left\{ \begin{array}{rcl}
 C I^{-2m} &  if & m \leq \ell\\
 0 &  if & m > \ell\\
 %2^m m!   f_m  I^{-2m}(s) +   O(I^{-2m -2}(s))       &  if & m = \ell,    \\
 %   O(I^{-2m -2}(s))      &  if & \ell >   m, \\
  %  0  &  if & \ell <   m. 
\end{array}    \right. 
 \end{equation}
\end{lemma}
\begin{proof} 
We note first that for all integer $j$

\[\begin{array}{lll} \langle  y^j, H_m   \rangle_{L^2_{\rho_s}} &=&\dsp I^{-m}\int y^j h_m(yI)\rho_s dy,\\
&=&\dsp I^{-m-j}\int z^j h_m(z)\rho dz\mbox{ where $z=Iy$}.\\
&=& CI^{-m-j}.
\end{array}
\]
then we conclude that 
\[\frac{ \langle  y^j, H_m   \rangle_{L^2_{\rho_s}}}{\|H_m\|_{L^2_{\rho_s}}}=C I^{m-j}\]
and
\[
   \frac{ | \langle  f, H_m   \rangle_{L^2_{\rho_s}}| }{\|H_m\|_{L^2_{\rho_s}}}\leq  
\left\{ \begin{array}{rcl}
% C          &  if & m = \ell,    \\
     C  &  if & \ell \geq   m, \\
    0  &  if & \ell <   m, 
\end{array}    \right. 
\]
which end the proof of Lemma. \end{proof}

\begin{lemma}\label{small-integral-y-ge-I-delta}
 Let us consider $m \in \N    $, $I(s)$ defined as in \eqref{defi-I-s}   and $f \in  L^\infty_K$ for some $K >0$, where $L^\infty_K$ is defined by \eqref{defi-L-M}. Then, we have 
\begin{eqnarray}
\left|  \int_{|y| \geq 1} f H_m \rho_s dy  \right|  \le C(m,K) \|f\|_{L^\infty_M} e^{-\frac{1}{8}I(s)}, \text{ with } s \ge s_0.
\end{eqnarray}
consequently, if   $f$ has the form
$$ f(y) = \sum_{j=0}^\ell f_j y^j, f_j \in \R,       $$
then, we have
\begin{equation}\label{integral-yle-I-delta-polynomial}
  \left|  \int_{|y| \le 1} f(y) H_{m}(y,s) \rho_s(y) dy \right| \leq  \left\{ \begin{array}{rcl}
  C I^{-2m}(s)     &  if &  \ell \geq m,    \\
   % O(I^{-2m -2}(s))       &  if & \ell >   m, \\
     C e^{-\frac{I(s)}{8}}  &  if & \ell <   m. 
\end{array}    \right.
\end{equation}
\end{lemma}
\begin{proof}
Let us decompose as follows
$$   \int_\R  f H_{m} \rho_s  dy =   \int_{|y| \le  1}  f H_{m} \rho_s  dy +  \int_{|y| \ge 1} f H_{m} \rho_s  dy.   $$

%\begin{eqnarray*}
%\left| \int_{y \ge I^{-\delta}(s)} f H_{2k} \rho_s  dy \right| \le C(\delta ) e^{-\frac{1}{2}I(s)} \|f\|_{L^\infty_M}.  
%\end{eqnarray*}
From \eqref{defi-rho-s}, we get
\begin{eqnarray*}
\left| H_{m}(y)   \right| \le C\left( 1 + y^{m} \right).
\end{eqnarray*}

Then, we get
\begin{eqnarray*}
\left| \int_{y \ge 1} f H_{m} \rho_s  dy \right| \le C  \|f\|_{L^\infty_K}  \int_{y \ge 1}  (1 + y^{m + K}) e^{-\frac{I^2(s)y^2}{4}} I(s) dy. 
\end{eqnarray*}
By changing  variable $z = I(s) y$, we obtain
\begin{eqnarray*}
\int_{y \ge 1}  (1 + y^{m + K}) e^{-\frac{I^2(s)y^2}{4}} I(s) dy & = & \int_{z \ge I(s)  } \left(1 + z^{m+K} I^{-m-K}(s) \right) e^{-\frac{z^2}{4}}  dz\\
& \le  & \int_{z \ge I(s)  } \left(1 + z^{2m+M} \right) e^{-\frac{z^2}{4}}  dz\\
& \le & e^{-\frac{I^{2}(s)}{8}} \int_0^\infty  (1 + z^{m+K})e^{-\frac{z^2}{8}} dz \\
& \le & C e^{-\frac{I(s)}{8}},
\end{eqnarray*}
which concludes the proof of the Lemma.
\end{proof}

\begin{lemma}[Some scaled Hermite polynomial identities]\label{Hermite_Identies}
Let us consider $ H_n, n \in \mathbb{N}$ defined as in \eqref{eigenfunction-Ls} be scaled Hermite polynomials  and $\ell \in  \mathbb{N}$, then, we have 
\begin{equation}\label{Hermite-identities}
     y^\ell H_n (y,s) =  \sum_{j=0}^{\left[ \frac{\ell+n}{2}\right]} c_{j,\ell,n} (s) H_{n+\ell - 2j}(y,s).
\end{equation}
In particular, when $\ell =1$ and $ \ell =2$, we have  
\begin{eqnarray}
y^2 H_{n}(y,s)  &=&    H_{n+2} (y,s) + (4n+2) I^{-2}(s) H_n(y,s) + 4 n(n-1) I^{-4}(s) H_{n-2}(y,s)  \label{Hermite-identities-ell=2}\\
 y H_{n-1} (y,s) &=& H_n(y,s) + I^{-2} (s) 2(n-1) H_{n-2}(y,s) \label{Hermite-identities-ell=1}.
\end{eqnarray}
\end{lemma}
\begin{proof}
The result immediately follows the fact that $\{H_n, n\ge 0\}$ is a basic of $L^2_\rho$ and  $y^\ell H_n(y,s)$ is a polynomial of order $n+\ell$.  In addition to that,  we also have the property that $\langle y^{\ell} H_{n}, H_k  \rangle_{L^2_{\rho}} =0$ whenever  $k + \ell < n $, then,  the terms of $H_{n+\ell-2j}, j > \left[ \frac{\ell+n}{2} \right]$ don't appear in the sum \eqref{Hermite-identities}. Thus, the result of  \eqref{Hermite-identities} completely follows.  In addition to that, we also specify the constants $c_{j,\ell,n}$ by 
$$ c_{i,\ell,n} = \frac{ \langle y^\ell H_{n} H_{n+\ell-2j}  \rangle_{L^2_\rho}}{\| H_{n+\ell-2j}\|^2_{L^2_\rho}}.$$
Finally, for the special cases $\ell=1,2$, we get \eqref{Hermite-identities-ell=2} and  \eqref{Hermite-identities-ell=1}, and we concludes the proof of the Lemma.
\end{proof}

\begin{lemma}\label{lemma-estimation-K-q-} It holds that 
\[\left |K_{\tau\sigma} q_-\right |_{\tau}\leq Ce^{-\frac{1}{p-1}(\tau-\sigma)}\left | q_-\right |_{\sigma}.
\]
\end{lemma}
\begin{proof}
 We proceed as in the proof of Lemma 1 in Section 4, page 569 from \cite{BKLcpam94}.
We start by recalling the change of variable $z=yI(s)$ and we note
\[\theta(z)=q_-(y)=q_-\left (\frac{z}{I}\right )\, \tilde \theta (z)=(\mathcal K_{\tau \sigma}\theta)\left (\frac{y}{I}\right ),\]
then we have
\[\tilde\theta=e^{(\tau-\sigma)\mathcal{L}}\theta,\]
where $e^{s\mathcal{L}}$ is the semigrpup generated by the operator $\mathcal{L} q=\Delta q-\frac{y}{2}q+q$
\[e^{s\mathcal{L}}(z,z')=\frac{1}{[4\pi (1-e^{-t})]^{\frac N2}}\exp \left [-\frac{(z'-ze^{-s/2})^2}{4(1-e^{-s})}\right ].\]

From the definition of the shrinking set $\mathcal{V}_{\delta,s}$ given by Definition \ref{definition-shrinking-set}, we have \[|\theta(z)|\leq I^{M}(1+|z|^M)|q_-|_{\sigma},\]
with 
\[(\theta,h_m)=\dsp\frac{\int \theta h_m(z)\rho (z)dz}{\|h_m\|^2} \mbox{  for } m\leq [M]\mbox{ and }\rho (z)=\frac{e^{-z^2/4}}{(4\pi)^{N/2}},\]
where $M=\frac{2kp}{p-1}$.\\
Proceeding as in the derivation of equation (66) in Section 3 and Lemma 4, in section 3 from \cite{BKLcpam94} and using Definition \ref{definition-shrinking-set} we get for $\tau-\sigma\geq 1$,
\[
\begin{array}{lll}
|\tilde \theta(z)| &\leq & Ce^{(\tau-\sigma)}e^{-([M]+1)\frac{\tau-\sigma}{2}}I^{-M}(1+|z|^M) |q_-|_{\sigma}\mbox{ for }0\leq m\leq [M]+1.
%&=&\dsp \left | \int   \right |  \\
%     & 
\end{array}
\]
We recall that $M=\frac{2kp}{p-1}$ and using the fact that

\[e^{-\frac{M}{2}(\tau-\sigma)}I^{-M}(\sigma)=e^{-\frac{p}{p-1}(\tau-\sigma)}I^{-M}(\tau), \]
then we obtain the desired result
\[|\mathcal{K}_{\tau\sigma} q_-(y)|\leq Ce^{-\frac{(\tau-\sigma)}{p-1}}(I^{-M}(\tau)+|y|^M)|q_-|_{\sigma}.\]
\end{proof}

\begin{lemma}\label{lemma-estimation-K-phi}
There exists a constant C such that if $\phi$ satisfies 
\[\forall x\in \mathbb{R},\;\; |\phi(x,\tau)|\leq (I^{-M}(\tau)+|x|^{M}),\]
then for all $y\in\mathbb{R}$ and $s\geq s_0$, we have
\[|K_{\tau\sigma} (\phi)(y,\tau)|\leq C e^{-\frac{(\tau-\sigma)}{p-1}}(I^{-M}(\tau)+|y|^{M}).\]
\end{lemma}

\begin{proof}
The proof is similar to the proof of Lemma \ref{lemma-estimation-K-q-} and we use the fact that for $z=yI(s)$, if 
\[|\Phi(z,\tau)|\leq (1+|z|^M),\]
then 
\[|e^{s\mathcal{L}}P_-(\Phi)(z)|\leq C e^{\frac{[M]+1}{2}s}(1+|z|^{M}).\]
\end{proof}

\medskip

\bibliographystyle{alpha}
\bibliography{mybib}

\end{document}